\documentclass[sn-basic]{sn-jnl}
\usepackage{manyfoot}
\usepackage{xcolor}
\usepackage{amsmath,amssymb,amsfonts}%

\usepackage[utf8]{inputenc}

\usepackage{graphicx}
\usepackage{mathtools}
\usepackage{amsfonts}
\usepackage[shortlabels]{enumitem}
\usepackage{color}
\usepackage{comment}
\usepackage{dsfont}
\usepackage[normalem]{ulem}
\usepackage{lmodern}
\usepackage{url}

\usepackage{ulem}

\usepackage{amsthm}
\newtheorem{theorem}{Theorem}[section]
\newtheorem{proposition}[theorem]{Proposition}
\newtheorem{lemma}[theorem]{Lemma}
\newtheorem{corollary}[theorem]{Corollary}
\newtheorem{remark}[theorem]{Remark}
\newtheorem{definition}[theorem]{Definition}
\newtheorem{assumption}[theorem]{Assumption}

\newcommand{\Ad}{\operatorname{Ad}}
\newcommand{\grad}{\operatorname{grad}}

\newcommand{\msf}{\mathsf}
\newcommand{\mf}{\mathfrak}
\newcommand{\mc}{\mathcal}
\newcommand{\mbb}{\mathbb}
\newcommand{\mbf}{\mathbf}
\newcommand{\mds}{\mathds}

\newcommand{\SU}{\mathrm{SU}}
\newcommand{\U}{\mathrm{U}}
\newcommand{\tr}{\operatorname{tr}}
\newcommand{\id}{\mds{1}}
\newcommand{\T}{{\sf T}}

\begin{document}

\title[Article Title]{Randomized Gradient Descents on Riemannian Manifolds: Almost Sure Convergence to Global Minima in and beyond Quantum Optimization}

\author[1,2]{\fnm{Emanuel} \sur{Malvetti}}\email{emanuel.malvetti@tum.de}

\author[3]{\fnm{Christian} \sur{Arenz}}\email{christian.arenz@asu.edu}

\author[4]{\fnm{Gunther} \sur{Dirr}}\email{dirr@mathematik.uni-wuerzburg.de}

\author[1,2]{\fnm{Thomas} \sur{Schulte-Herbrüggen}}\email{tosh@tum.de}

\affil[1]{\orgdiv{School of Natural Sciences}, \orgname{Technische Universit\"at M\"unchen}, \orgaddress{\city{Garching}, \postcode{85748}, \country{Germany}}}

\affil[2]{\orgname{Munich Center for Quantum Science and Technology (MCQST) \& Munich Quantum Valley (MQV)}, \orgaddress{\city{M\"unchen}, \postcode{80799}, \country{Germany}}}

\affil[3]{\orgdiv{School of Electrical, Computer and Energy Engineering}, \orgname{Arizona State University}, \orgaddress{\city{Tempe}, \postcode{85287}, \state{Arizona}, \country{USA}}}

\affil[4]{\orgdiv{Institute of Mathematics}, \orgname{University of W\"urzburg}, \orgaddress{\city{W\"urzburg}, \postcode{97074}, \country{Germany}}}

\abstract{
We analyze convergence of gradient-descent methods on Riemannian manifolds.  
In particular, we study randomization of Riemannian gradient algorithms for minimizing smooth cost functions (of Morse--Bott type). 
We prove that randomized gradient descent methods, where the Riemannian gradient is replaced by a random projection of it,
converge to a single local optimum almost surely despite the existence of saddle points. 
We consider both uniformly distributed and discrete random projections.
We also discuss the time required to pass a saddle point.
As a major application, we consider ground-state preparation through quantum optimization over the unitary group. 
In mathematical terms our randomized algorithm applied to the trace function $U \to \tr(AU\rho U^*)$ almost surely converges to its {\em global minimum}. 
The minimum corresponds to the smallest eigenvalue (ground state) of the selfadjoint operator $A$ (Hamiltonian) if $\rho$
is a rank-one projector (pure state). 
In this setting, one can efficiently replace the uniform random projections by implementing so-called discrete unitary 2-designs.}


\maketitle

\section{Introduction} \label{sec:intro}

Both gradient systems and numerical linear algebra come with a long history, as well as a vibrant research activity. 
For instance, gradient systems have been advanced lately from Hilbert manifolds~\cite{NeubergerLNM2010} to general metric spaces (e.g., in the context of optimal mass transport~\cite{ambrosio2021lectures, mielke2023introduction}). 
In contrast, numerical linear algebra has faced factorization problems of large tensor structures~\cite{batselier2018computing} (tensor SVD) using ``randomization strategies'' for their classical algorithms in order to handle these huge amounts of data.
A few decades ago, a systematic study of their interplay started by analyzing the QR-algorithm and interior point methods from a (Riemannian) geometric point of view, cf.\ the collection of~\cite{Bloch94}.
In this development, the seminal work of~\cite{Bro88+91, Bro89} on the so-called ``double-bracket flow'' on orbits of the orthogonal group (and likewise of the unitary group in~\cite{Bro93}), followed by independent similar ideas of~\cite{ChuDriessel90}, has sparked a plethora of applications.
For an early overview on optimization techniques on Riemannian manifolds (including higher-order methods) see~\cite{Smith94} and \cite{Bloch94}.
A detailed mathematical account can be found in the monograph by~\cite{HM94}, including discretization schemes, suitable step sizes, and proofs of convergence. 
For the latter, the authors heavily exploited the Morse--Bott structure of the respective cost functions (i.e.~the particular structure of the Hessian at possibly non-isolated critical points, cf.\ Subsec.~\ref{subsec:Morse-Bott} here and~\cite{Duistermaat83}) to obtain convergence to a single critical point.
Based on these developments, higher-order methods with applications to a variety of cost functions on classical matrix manifolds are treated in~\cite{AMS08}
or, more recently, by~\cite{Boumal23}.
In~\cite{SGDH08}, the above ideas have been extended to flows and their discretization schemes on more general homogeneous spaces. 
There, the authors also established connections to applications in quantum dynamics as well as to applications in numerical (multi-)linear algebra, such as higher-rank tensor approximations --- as elaborated by~\cite{CDH2012} to complement standard methods (such as higher-order powers of~\cite{LMV-II} or quasi-Newton methods of~\cite{SL10}). 

Independently, in the physics community, \cite{Weg94} (re)devised gradient flows of Brockett type, e.g., to (band-)diagonalize Hamiltonians, which the monograph by~\cite{Kehr06} elaborated to address further quantum many-body applications. 

Often, gradient-flow approaches for deriving optimization schemes on Riemannian manifolds hinge on the computability of the Riemannian exponential, ensuring to take the gradient from the tangent space back to the manifold. 
To be precise, this is crucial whenever the manifold cannot be identified with its tangent spaces, such as for the unitary group and its orbits.
But, in particular in this case, analytical expressions and
efficient numerical approximations for the Riemannian exponential are well known.

With these stipulations, Riemannian optimization constitutes a versatile toolbox, since  cost functions can readily be tailored to the particular application such as finding extremal eigen- or singular values.
These instances also connect to other branches of mathematics such as numerical and $C$-numerical ranges and their extremal values, the numerical and $C$-numerical radii (see~\cite{GR-97, Li94}). 

\smallskip

{Over the last decade, techniques from Riemannian optimization have also been adopted in quantum information science.
For example, in quantum computing, hybrid quantum-classical algorithms \cite{bharti2022noisy}, such as variational quantum
algorithms (VQAs) \cite{cerezo2021variational, tilly2022variational}, have been developed to solve ground-state problems
\cite{peruzzo2014variational},
combinatorial optimization problems \cite{farhi2014quantum}, and quantum machine learning problems \cite{biamonte2017quantum}. 
\cite{wiersema2023optimizing} proposed projecting the Riemannian gradient into smaller dimensional subspaces
to obtain scalable\footnote{Sec.~\ref{sec:gs-opt} gives a precise meaning of \/`scalable quantum computers\/'} quantum-computer implementations. However, the 
``dimension reduction'' strategy 
comes
at a cost:
{While 
the full gradient flow converges almost surely to 
some local minimum
in spite of strict saddle points\footnote{i.e., saddle points at which the
Hessian has at least one negative eigenvalue} \cite{Lee19}, for the projected versions no such guarantees can be made.
Indeed, numerical simulations see convergence
to artificial local minima induced by the projection
into a {\em fixed} subspace.} 
Inspired by the success of classical randomized gradient descent \cite{ruder2016overview}, the 
problem {of artificial minima} has recently been {circumvented} by some of the authors in \cite{magann2023randomized} via ``random projections''
of the Riemannian gradient.
(Details on different
optimization algorithms and their efficient implementations on {\em quantum} computers are postponed to Section \ref{sec:gs-opt}.)}

\smallskip

In this paper we consider the randomized gradient algorithm presented in~\cite{magann2023randomized} in the more general setting of {the large class\footnote{{Indeed, on a compact manifold, the set of Morse functions, which is a subset of the Morse--Bott functions, forms an open dense subset of the space of smooth functions~\cite{Audin14}. In this sense, Morse--Bott functions are generic.}} of Morse--Bott cost functions} on Riemannian manifolds.
For both continuous and discrete probability distributions of the random gradient directions, we prove that in fact the algorithm almost surely converges to a local minimum.
{This is a significant improvement on} the state-of-the-art, where thus far (see~\cite{gutman}) only convergence to the {\em set} of critical points has been proven without the Morse--Bott assumption.
{Related papers have established convergence to local optima for perturbed and stochastic gradient descent~\cite{JinStoch21}, also generalized to Riemannian manifolds~\cite{criscitiello2019} and~\cite{Sun19}.}
Finally, we show how our result can be applied to quantum optimization tasks, such as ground state preparation: 
{It is by construction that} in the specific quantum setting, there are no local optima, whereby we prove that our randomized gradient algorithm then converges almost surely to {a single {\em global} minimum}.  

\begin{figure}[htb]
\centering
\includegraphics[width=0.55\textwidth]{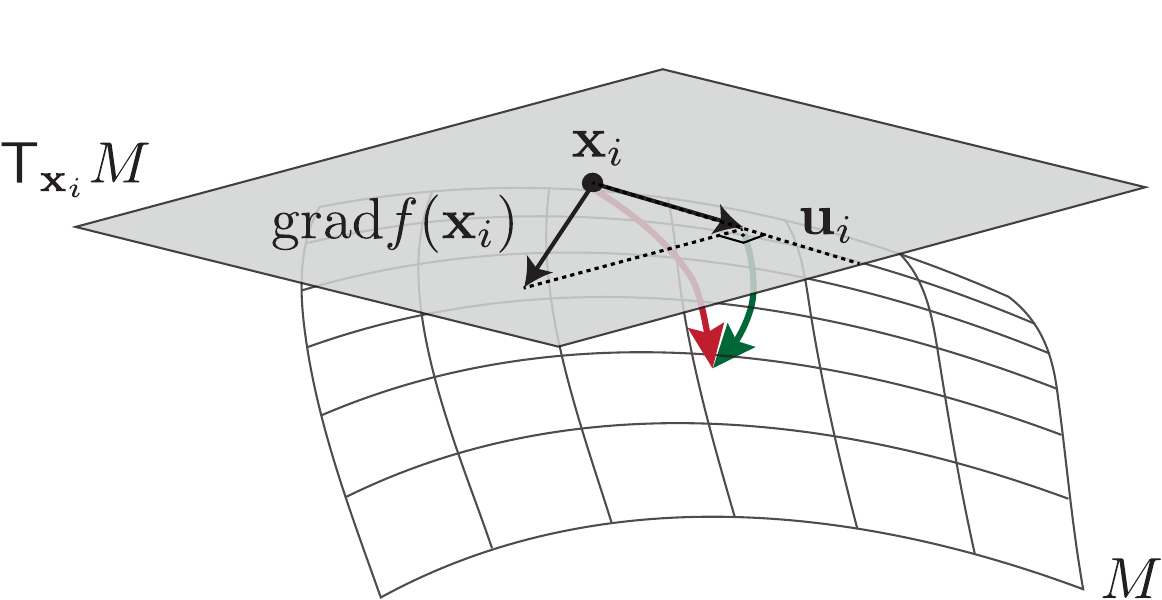}
\caption{Schematic representation of a randomized gradient descent algorithm on a Riemannian manifold $M$. 
At a point $\mbf x_i$ on $M$, the Riemannian gradient $\operatorname{grad} f(\mbf x_i)$ of a cost function $f:M\to\mathbb R$ lies in the tangent space $\T_{\mbf x_i} M$ depicted as the plane spanned by $\mbf u_i$ and its orthocomplement.
Instead of iteratively following the full gradient in order to optimize $f$, here in each step we consider a projection of $\text{grad}f(\mbf x_i)$ onto a randomly chosen tangent-space direction $\mbf u_i$. 
This is key for {\em efficient} implementations of gradient flows in large-scale quantum optimization problems.
In the case of the special unitary group the randomized directions are given by (traceless) skew-Hermitian  operators $\mbf u_{i}=i H_{i}$. 
The corresponding quantum circuit consists of a sequence of unitary transformations as shown in Fig.~\ref{fig:intro2}, where the cost function is then denoted by $J$. 
}
\label{fig:intro1}
\end{figure}

\subsubsection*{Structure and Main Results}

In Sec.~\ref{sec:Riem-Grad-Flows} we introduce the fundamental notions from Riemannian geometry to describe gradient flows (Sec.~\ref{subsec:basic-concepts}), recollecting some well-known results on the convergence of such gradient flows (Secs.~\ref{sec:gradient-flows} and~\ref{subsec:Morse-Bott}), and their algorithmic implementation as gradient descents (Sec. \ref{sec:num-grad-descent}).
In Sec.~\ref{sec:algorithm-convergence} we describe a gradient descent algorithm where, in each step, the full gradient of a high-dimensional problem is projected on just a randomly chosen direction of the tangent space {as illustrated in Fig.\ref{fig:intro1}}.
{We consider both uniformly distributed\footnote{with respect to the Haar measure} and discrete random directions.}

The main results are obtained in Secs.~\ref{sec:escape-saddles} and \ref{sec:almost-sure-convergence}, where we prove that a randomly projected gradient descent algorithm {of} a smooth Morse--Bott cost function (with compact sublevel sets) on a Riemannian manifold comes with two vitally beneficial properties:
{\em (i) it almost surely escapes saddle points} --- as is shown in Lem.~\ref{lemma:saddle-as-gen}, and as a consequence
{\em (ii) it almost surely converges to a {single} local minimum} --- as shown in Thm.~\ref{thm:almost-sure-convergence}.

In Sec.~\ref{sec:low-dims}, the case of a two-dimensional saddle point is {studied
using analytical methods and 
numerical simulation}. We obtain good approximations for the time necessary to pass the saddle point.

Finally, Sec.~\ref{sec:gs-opt} shows how the previous results can be applied in quantum optimization: {by construction, the important} problem {class} of ground-state preparation
{comes with the particular advantage that our algorithm} achieves almost sure convergence to a single
{\em global} minimum. {Moreover, in these cases the algorithm can be efficiently implemented on quantum computers}. 

\section{Riemannian Gradient Flow and Descent}
\label{sec:Riem-Grad-Flows}

\subsection{Basic Concepts}
\label{subsec:basic-concepts}

First, we recall some basic notions and notations from Riemannian geometry.
Let $M$ denote a finite dimensional smooth manifold of dimension $N$ with tangent and cotangent bundles $\T M$ and $\T^*M$, respectively. 
Moreover, let $M$ be endowed with a Riemannian metric $\msf g$, i.e.\ a smoothly varying scalar product $\msf g(\mbf x) = \langle \,\cdot, \cdot\,\rangle_{\mbf x}$ on each tangent space $\T_{\mbf x} M$, $\mbf x\in M$. 
The pair $(M,\msf g)$ is called a \emph{Riemannian manifold}.
{For ease of notation we will usually omit the subscript in the scalar product and write $\langle \,\cdot, \cdot\,\rangle$ and similarly $\|\cdot\|$ for the resulting norm.}

\medskip
Every Riemannian manifold is equipped with a Riemannian density $\mu$ induced by the metric $\msf g$, see~\cite[Prop.~2.44]{Lee13}. 
This density can be used to integrate functions, to induce a measure on $M$ and to define sets of measure zero. 
Equivalently (but without introduction of a Riemannian density), one could say that $S \subset M$ has measure zero if there is an atlas
such that $S \cap {\rm dom} \,\sigma$ has Lebesgue measure zero in every chart $\sigma$. 
In particular, any submanifold of dimension strictly smaller than $N$ has measure zero.

\subsection{Gradient Flows and Asymptotic Behaviour (I)} \label{sec:gradient-flows}

In the following general description, we depart from the above notation of \/`the\/' cost function $J$ borrowed from quantum
optimization and introduce $f : M \to \mathbb{R}$ for an arbitrary smooth (cost) function on $M$ with differential $df:M\to\T^* M$.
Then the gradient of $f$ at $\mbf x\in M$, denoted by $\grad f(\mbf x)$, is the unique vector in $\T_{\mbf x}M$ determined by the identity
\begin{equation}
\label{eq:gradient}
df(\mbf x)\xi = \langle\grad f(\mbf x),\xi \rangle
\end{equation} 
for all $\xi \in \T_{\mbf x} M$.
Equation \eqref{eq:gradient} naturally defines a vector field on $M$ via
$\grad f:M\to\T M$, 
$\mbf x\to\grad f(\mbf x)$
called gradient vector field of $f$. The corresponding ordinary differential equation
\begin{equation}
\label{eq:gradient-flow}
\dot{\mbf x}=-\grad f(\mbf x) =: F(\mbf x)
\end{equation}
and its flow are referred to as the \emph{gradient system} and the \emph{gradient flow} of $f$, respectively. 

In local coordinates, the above reads as follows:
Let $\sigma:U\to\mathbb R^N$ be a chart about $\mbf x\in M$, denote $\mbf y=\sigma(\mbf x)$, and let $\tilde f:=f\circ\sigma^{-1}$ be the chart representation of $f$. 
Moreover, let $\tilde{\msf g}:=\sigma_*\msf g$ be the chart representation (where $\sigma_*$ denotes the pushforward) of the metric $\msf g$ on $\sigma(U)$, i.e.\ $\tilde {\msf g}(\mbf y) = \sum_{i,j=1}^n\tilde{\msf g}_{ij}(\mbf y)\,dy^idy^j$, where $\tilde{G}(\mbf y) :=\big(\tilde{\msf g}_{ij}(\mbf y)\big)$ is a positive definite matrix varying smoothly in $\mbf y\in \sigma(U)$. 
Finally, let $\tilde{F} := \sigma_*F$ be the chart representation of the gradient vector field $-\grad f$. 
Then it holds that
$$
\tilde F^j(\mbf y) 
=-\sum_{i=1}^n \tilde{\msf g}^{ij}(\mbf y) \frac{\partial \tilde{f}}{\partial y_i}(\mbf y),
$$
where $\tilde{\msf g}^{ij}(\mbf y)$ are the entries of the inverse of $\tilde{G}(\mbf y)$ and $\frac{\partial}{\partial y_i}$ denotes the $i$-th partial derivative.

\medskip
Any convergence analysis of gradient systems is based on the following two observations: (i) the critical points of $f$ are the equilibria of \eqref{eq:gradient-flow}, (ii) $f$ constitutes a Lyapunov function for
\eqref{eq:gradient-flow}, i.e.~$f$ is monotonically decreasing along solutions.
Despite their simplicity, they immediately allow several non-trivial statements about the asymptotic behavior of~\eqref{eq:gradient-flow}, which can be characterized by its $\omega$-limit sets
$$
\omega(\mbf x_0) := \bigcap_{s > 0}
\overline{\big\{\varphi(t,\mbf x_0) \;|\; t>s\big\}}\,,
$$
where $ t \mapsto \varphi(t,\mbf x_0)$ denotes the unique solution of
\eqref{eq:gradient-flow} with initial value $\varphi(0,\mbf x_0)=\mbf x_0$.

\medskip
\begin{proposition}[\cite{Lee13}]
\label{prop:gradient-flow}
If f has compact sublevel sets, i.e.\ if the sets 
$\{\mbf x \in M : f(\mbf x) \leq c\}$ are compact for all $c \in  \mathbb R$, then every trajectory of~\eqref{eq:gradient-flow} converges to the set of critical points.
More precisely, every trajectory of~\eqref{eq:gradient-flow} exists for all $t\geq0$ and its $\omega$-limit set is a non-empty, compact and connected subset of the set of critical
points of $f$.
\end{proposition}

\medskip
Although Prop.~\ref{prop:gradient-flow} shows that $\omega(\mbf x_0)$ is contained in the set of critical points of $f$, it does not guarantee convergence to a single critical point and, indeed, there are smooth gradient systems whose trajectories exhibit a non-trivial convergence behavior to the set of critical points (cf.~``Mexican hat counter-example'' by~\cite{Curry44}).
Yet for isolated critical points, one has the following trivial consequence of Prop.~\ref{prop:gradient-flow}.

\medskip
\begin{corollary}
\label{cor:iso-critical-point}
If f has compact sublevel sets and all critical points are isolated, then any solution of~\eqref{eq:gradient-flow} converges (for $t \to \infty$) to a single critical point of $f$. Moreover, if additionally all saddle points are strict,\footnote{see Subsec.~\ref{subsec:Morse-Bott} below} then for almost all initial points 
(in the sense of Sec.~\ref{subsec:basic-concepts}),
the flow converges to a local minimum. 
\end{corollary}
\medskip
\noindent
Continua of critical points are much more subtle to handle. 
Some enhanced conditions guaranteeing convergence to a single critical point will be briefly discussed in the next subsection.

\subsection{The Hessian, Morse--Bott Functions and Asymptotic Behaviour (II)}
\label{subsec:Morse-Bott}

As in the Euclidian case, linearizing the vector field $F$ at its equilibria sheds light on its local stability. 
Clearly, since $F = -\grad f$ constitutes a gradient vector field,
its linearization at an equilibrium $\mbf x\in M$ is given by the Hessian $\operatorname{H}_f(\mbf x)$ of $f$ at $\mbf x\in M$.
In general, if $M$ is non-Euclidian, the computation of $\operatorname{H}_f(\mbf x)$ can be rather involved.
Yet, at critical points $\mbf x_0 \in M$ of $f$, the \emph{Hessian} is given by the symmetric bilinear form
\begin{equation}
\label{eq:hessian-1}
\operatorname{H}_f(\mbf x_0)(\xi,\eta) 
:= 
\sum_{i,j=1}^n\operatorname{\tilde{H}}(f\circ\sigma^{-1})
\big(\sigma(\mbf x_0)\big)_{ij}
(D\sigma(\mbf x_0)\xi)_i
(D\sigma(\mbf x_0)\eta)_j\,,
\end{equation}
where $\sigma:U\to\mbb R^n$ is any chart around $\mbf x_0\in M$ and $\operatorname{\tilde{H}}(f\circ\sigma^{-1})$ denotes the ordinary Hesse
matrix of the chart representation
$f\circ\sigma^{-1}$. 
It is straightforward to show that \eqref{eq:hessian-1} is independent of $\sigma$.
Moreover, we will call $\mbf x_0$ a \emph{strict saddle point} if $\operatorname{H}_f(\mbf x_0)$
(or more precisely the associated symmetric operator) has at least one negative eigenvalue. 

The above concepts allow a trivial generalization of a well-known result from elementary calculus which follows straightforwardly in local coordinates.

\medskip
\begin{proposition}
\label{prop:hessian-minimum}
Let $M$ be a Riemannian manifold and let $\mbf x$ be a critical point of the smooth function $f:M\to\mathbb R$. 
If $\operatorname{H}_f(\mbf x)$ is positive definite, then $\mbf x$ is a strict local minimum of $f$.
\end{proposition}

\medskip
\noindent
Now the question arises whether the (asymptotic) stability of an equilibrium $\mbf x\in M$ of \eqref{eq:gradient-flow} may dependent on the Riemannian metric $\msf g$ --- and the answer is surprisingly ``yes'', cf.~\cite{Takens71}.
However, certain properties such as being a strict local minimum or an isolated critical point are obviously not up to the choice of the metric, and thus the (asymptotic) stability of those equilibria is  independent of the Riemannian metric as the next theorem shows.

\medskip
\begin{theorem} \label{thmI:pointconvergence}
Let $M$ be a Riemannian manifold and $f:M\to\mathbb R$ a smooth function. Then the following hold:
\begin{enumerate}[(i)]
\item
Every strict local minimum of $f$ is a stable equilibrium of \eqref{eq:gradient-flow}.
\item
Every strict local minimum of $f$ which is additionally an isolated critical point is an asymptotically stable equilibrium
of \eqref{eq:gradient-flow}.
\end{enumerate}
\end{theorem}

\medskip
\noindent
Both assertions follow immediately from classical stability theory
by taking $f$ as Lyapunov function, cf.~\cite{HM94,Irwin}.
Handling non-isolated critical points is much more subtle. 
A first hint on how to approach this issue is obtained by Cor.~\ref{cor:iso-critical-point} which could be restated (in a slightly weaker version) as follows: 

\medskip
\noindent
{\it For Morse functions with 
compact sublevel sets every solution of the corresponding gradient system converges (for $t \to \infty$) to a single critical point.} 

\medskip
\noindent
This suggests to work with Morse--Bott functions when it comes to non-isolated critical points. 
Recall that a smooth function $f:M\to\mathbb R$ is a called \emph{Morse--Bott function}, if the set $C$ of critical points is a closed submanifold of $M$, such that its tangent space $\T_{\mbf x} C$ coincides with the kernel of the Hessian operator for all $\mbf x\in C$. 
Note that $C$ is allowed to have several
connected components with possibly different dimensions,~\cite[p.~366]{HM94},
\cite[Def.~2.41]{Nicolaescu11} or 
\cite{Hurtubise2010}.
Thus Morse functions are particular Morse--Bott functions, where $C$ consists only of $0$-dimensional mainfolds (i.e.\ isolated points).

Now, the above concept allows a generalization of the Morse--Palais Lemma to Morse--Bott functions which is often called Morse--Bott Lemma, see
\cite{Hurtubise2004a}. 
It yields
a local normal form of Morse--Bott functions near their critical points, which reads in local coordinates as follows
\begin{equation*}
\tilde{f}(\mbf x,\mbf y,\mbf z) = \|\mbf x\|^2 - \|\mbf y\|^2 \,,
\end{equation*}
with $\mbf x\in\mathbb R^{n_+}$, $\mbf y\in\mathbb R^{n_-}$, and $\mbf z\in\mathbb R^{n_0}$, where $n_+$, $n_-$ and $n_0$ are the numbers of positive, negative and zero eigenvalues of the Hessian operator respectively.
Here $\|\cdot\|^2$ denotes the usual Euclidean norm.
This finally allows to prove that solutions of the respective gradient system converge to a single critical point.

\medskip
\begin{theorem}[Thm. 2.3 in \cite{aulbach1984}]
\label{thm:Morse--Bott}
Let $f:M \to \mathbb R$ be a Morse--Bott function on a Riemannian manifold $M$ with compact sublevel sets. Then every solution of the  gradient flow \eqref{eq:gradient-flow} converges to a single critical point. Moreover, for almost every initial point, the flow converges to a local minimum.
\end{theorem}

\medskip
Finally, we recall another very powerful result for analyzing the convergence of gradient systems which is based on \L{}ojasiewicz's celebrated gradient estimate~\cite{Loja84}.
Let $f:M \to \mathbb R$ be real analytic, ${\mbf x}_0 \in M$ a critical point of $f$, and assume without loss of generality that $f({\mbf x}_0) = 0$. 
Then, near ${\mbf x}_0$, one has the estimate
\begin{equation}
\label{eq:Lojasievicz}
\|\grad (\mu \circ f)({\mbf x})\| \geq c\,,
\end{equation}
where $\mu: \mathbb R^+ \to \mathbb R^+$ is a strictly increasing $C^1$-function and $c > 0$ some positive constant, cf.~\cite[Cor.~1.1.25]{Diss-Lageman}. In the literature, one usually finds $\mu(r) := r^{1-\theta}$ with $\theta \in (0,1)$.
Eq.~\eqref{eq:Lojasievicz} 
easily allows to bound the 
length of any trajectory of \eqref{eq:gradient-flow}
whose $\omega$-limit set is 
non-empty. Hence one gets the following result.

\medskip
\begin{theorem}[\cite{Loja84}] 
\label{thm:Lojasiewicz}
Let $(M,\mathrm g)$ and $f:M \to \mathbb R$ be are real analytic. Then every non-empty
$\omega$-limit set of \eqref{eq:gradient-flow} consists only of a single critical point.
\end{theorem}

\subsection{The Exponential Map and Numerical Gradient Descent}
\label{sec:num-grad-descent}

Finally, we approach the problem of discretization of \eqref{eq:gradient-flow} resulting in a convergent gradient descent method (cp.\ Fig.\ref{fig:intro1}). The ideas presented here can be traced back to \cite{Bro88+91, Bro89}
and \cite{Diss-Smith,Smith94}.
Let
\begin{equation} \label{eq:exponential-map}
\exp_{\mbf x}:\T_{\mbf x}M\to M
\end{equation}

\medskip
\noindent
denote the \emph{Riemannian exponential map} at
${\mbf x}\in M$, i.e. $t \to \exp_{\mbf x} (t\xi)$ is the unique geodesic with initial value ${\mbf x}\in M$ and initial velocity $\xi\in\T_{\mbf x} M$. 
Here, we assume for simplicity that $(M,\msf g)$ is (geodesically) complete, i.e.~\eqref{eq:exponential-map} is well-defined for the entire tangent space $\T_{\mbf x}M$. 
Probably the simplest discretization scheme
given by
\begin{equation}
\label{eq:gradient_descent}
{\mbf x}_{k+1} := \exp_{{\mbf x}_k} (-\eta_k \grad f ({\mbf x}_k)),
\end{equation}
can be seen as a ``natural'' generalization of the explicit Euler method. 
Here $\eta_k>0$ denotes an appropriate ``step size''\footnote{Note that the ``actual'' step size results form the modulus of $\eta \grad f({\mbf x}_k)$.} which may depend on $k\in\mathbb N$.

\medskip
\begin{remark}
It should be mentioned that gradient descent algorithms are usually studied as exact algorithms, not as numerical algorithms where real numbers are represented as floating point values and arithmetic is not exact. Numerical effects can be important, for instance cancellation effects from computing gradients in a naive way~\cite{muller18}. Nevertheless, this paper will ignore these numerical issues.
\end{remark}

\medskip
In order to guarantee convergence of \eqref{eq:gradient_descent} to the set of critical points, it is sufficient to apply the Armijo rule, see \cite{LueYe08}. An alternative to Armijo’s rule provides the step-size selection suggested by \cite{Bro93}, see also~\cite{HM94}. 
Moreover, for compact Riemannian
manifolds even a sufficiently small constant step size $\eta>0$ guarantees convergence:

\medskip

\begin{theorem}
If $f$ has compact sublevel sets, every trajetory of the discretized gradient descent~\eqref{eq:gradient_descent} (with constant but small enough step size) converges to the set of critical points.
Moreover, if additionally
\begin{enumerate}[(i)]
\item all critical points are isolated, then any trajectory of~\eqref{eq:gradient_descent} converges to a single critical point of $f$.
\item all saddle points are strict, then the set of initial points converging to strict saddle points has measure zero (in the sense of Sec.~\ref{subsec:basic-concepts}).
\end{enumerate}
\end{theorem}

\medskip
\noindent
Part (ii) of the above statement can be found in \cite[Cor.~6]{Lee19}.%
\footnote{There the authors consider Riemannian manifolds embedded in the Euclidean space and retractions instead of an intrinsic Riemannian exponential function, but this does not restrict the generality of the result, cf. the Nash embedding theorem~\cite[Thm.~46]{Berger03}.}
Thus, for a function with compact sublevel sets,
gradient descent \eqref{eq:gradient_descent} 
behaves similarly to its continuous counter-part, cf.~Prop.~\ref{prop:gradient-flow} and Cor.~\ref{cor:iso-critical-point}, i.e.\ it converges almost surely to a local minimum if the step size is chosen small enough.

\medskip
Deeper results, which yield convergence to a single critical point, are more subtle to derive. Here we present one result in this direction which is again based on the analyticity of the cost function $f$
and on \L{}ojasiewicz's inequaliy.

\medskip
\begin{theorem}\cite{Diss-Lageman}
If $(M,\msf g)$ and $f$ are real analytic, and the step sizes are chosen according to a version of the first Wolfe--Powell condition for Riemannian manifolds, then pointwise convergence holds. 
\end{theorem}

\medskip
Finally, in order to determine the largest admissible step size of our algorithm, we need a notion of Lipschitz continuity of vector fields on Riemannian manifolds.
Care has to be taken here since tangent vectors in different tangent spaces cannot be compared by default. Certainly, one can define local Lipschitz
continuity via local charts but this does not allow to assign a meaningful Lipschitz constant to the vector field. A natural and intrinsic way to do this is to define Lipschitz continuity via parallel transport along the unique connecting geodesic, as is done in~\cite{Fetecau22}. 

\medskip

\begin{definition}
Let $(M,\msf g)$ be a complete Riemannian manifold and let $X$ be a vector field on $M$.
We say that $X$ is \emph{$\ell$-Lipschitz} if for all $\mbf x,\mbf y\in M$ with%
\footnote{Here $r_M$ stands for the injectivity radius of $M$.}
$d(x,y)\leq r_M$ it holds that%
\footnote{Here $\Pi_{\mbf y,\mbf x}$ denotes the parallel transport from $\mbf y$ to $\mbf x$ along the unique length minimizing geodesic connecting the two points.}
$\|X(\mbf x)-\Pi_{\mbf y,\mbf x}X(\mbf y)\|_{\mbf x}\leq\ell\, d(\mbf x,\mbf y)$.
If $f:M\to \mbb R$ is a differentiable function, then we say that $f$ is \emph{$\ell$-smooth} if $\grad(f)$ is $\ell$-Lipschitz.
\end{definition}

\medskip
\noindent Note that this definition is slightly broader than the definition given in~\cite{Fetecau22}.

\section{Algorithm and Convergence} \label{sec:algorithm-convergence}

We will study the following randomly projected gradient descent algorithm {(cp.\ Fig.~\ref{fig:intro1})}. 
Given a Riemannian manifold $M$ of dimension $N$, an initial state $\mbf x_0\in M$, a smooth function $f:M\to\mathbb R$, and a step size $\eta>0$ (with upper bound to be determined), the update rule is given by 
\begin{align} \label{eq:alg}
\mbf x_{i+1}=\exp_{\mbf x_i}(-\eta g(\mbf x_{i},\mbf u_{i})),
\,\,\,
g(\mbf x,\mbf u) = \langle \mbf u,\grad f(\mbf x)\rangle \,\mbf u,
\,\,\,
\msf S_{\mbf x_i}M \ni
\mbf u_{i}\overset{\mathrm{i.i.d.}}{\sim}
\mathcal P(\mbf x_i),
\end{align}
where $\msf S_{\mbf x}M$ denotes the unit sphere in the tangent space $\msf T_{\mbf x}M$ and $\mc P(\mbf x_i)$ some probability distribution on the unit sphere $\msf S_{\mbf x_{i}}M$.
(Recall that ``i.i.d.'' stands for ``independent and identically distributed'').
Intuitively, the gradient is projected onto a randomly chosen direction $\mbf u$ at each step.
Throughout this paper we will consider two cases: \smallskip
\begin{enumerate}[(I)]
\item\label{it:cont} Either $\mc P(\mbf x_i)=\mathcal{U}(\msf S_{\mbf x_i}M)$ denotes the uniform (rotationally invariant) probability distribution on the unit sphere, also called Haar measure;
\item\label{it:fin} or $\mc P(\mbf x_i)=\mathcal{D}(\msf S_{\mbf x_i}M)$ denotes a finite probability distribution on the unit sphere {accessing all possible directions
in the sense of Assumption \ref{assump}}.
\end{enumerate}

\medskip
\begin{assumption}[Discrete case]\label{assump}
{Given} $k$ continuous vector fields $\xi_1,\ldots,\xi_k$
and {$k$ continuous weights $p_1,\ldots,p_k$ on $M$ with 
$p_j\geq0$ and \mbox{$\sum_{j=1}^k p_j\equiv 1$}}, 
{such that} $p_j\xi_j/\|\xi_j\|$ is continuous, where the expression is set to zero when it is undefined.
This implies for instance that $p_j(\mbf x)=0$ whenever $\xi_j(\mbf x)=0$.

Then $\mbf u_i$ takes value $\xi_j(\mbf x_i)/\|\xi_j(\mbf x_i)\|$ with probability $p_j(\mbf x_i)$. 
Moreover, we {suppose} that at every point $\mbf x$, the tangent vectors $p_1(\mbf x)\xi_1(\mbf x), \dots, p_k(\mbf x)\xi_k(\mbf x)$ span the entire tangent space even after removing any single vector $p_j(\mbf x)\xi_j(\mbf x)$ and, if present, all its scalar multiples.
\end{assumption}

\medskip
The following technical lemma will be useful later.

\medskip
\begin{lemma} \label{lemma:overlap}
Let $K$ be a compact subset of a manifold $M$, and let $\{v_j\}_{j=1}^k$ be a set of continuous vector fields on $K$ which span the tangent space at every point.
{Then, for any continuous, non-vanishing vector field $v$ on $K$ there is 
some $\varepsilon>0$ such that}
$$
{\min_{\mbf x\in K} \max_{j=1,\ldots,k} |\langle v_j(\mbf x),v(\mbf x)\rangle| = \varepsilon.}
$$
\end{lemma}

\begin{proof}
The functions $\mbf x\mapsto|\langle v_j(\mbf x),v(\mbf x)\rangle|$ defined on $K$ are continuous and non-negative.
Hence their maximum over $j$ is still a continuous non-negative function on $K$.
Since the $v_j$ span each tangent space, this function is even strictly positive, and by compactness of $K$, there is a positive global minimum.
\end{proof}

Note that a related condition is used in~\cite[Thm. 5.2]{gutman} to guarantee convergence to the {\em set} of critical points. 

\medskip
\noindent
Our goal is to analyze the convergence behavior of this algorithm, and in particular, to show that it converges almost surely to a local minimum of $f$.
A deterministic version of the above described algorithm {(in the sense that all ``perturbations'' from the classical gradient descent are deterministic and not random)} 
was analyzed in \cite[Def.~4.2.1, Thm.~4.3.1]{AMS08} under the heading ``gradient-related methods''.

\subsection{Basic Properties} \label{sec:alg-basics}

We start with some simple properties of the function $g$ defined in~\eqref{eq:alg}.

\medskip
\begin{lemma} \label{lemma:g-props}
Let $\mbf x\in M$ be given. 
If $\mbf u\in\msf S_{\mbf x}M$, 
then it holds that
$$
\langle\grad f(\mbf x),g(\mbf x,\mbf u)\rangle
=
\|g(\mbf x,\mbf u)\|^2.
$$
\end{lemma}

\begin{proof}
Using the definition of $g(\mbf x,\mbf u)$ and the fact that $\mbf u$ is a unit vector, we immediately obtain
$
\langle\grad f(\mbf x),g(\mbf x,\mbf u)\rangle
=
\langle \mbf u,\grad f(\mbf x)\rangle^2
=
\|g(\mbf x,\mbf u)\|^2.
$
\end{proof}

\begin{corollary} \label{coro:g-props}
Note that if $\mbf u$ is {uniformly distributed (w.r.t.~the Haar measure)} on the sphere, then
$$
\langle \mbf u,\grad f(\mbf x)\rangle^2 
\overset{d}{=}
u_N^2 \|\grad f(\mbf x)\|^2,
$$
where $\overset{d}{=}$ means that the random variables have the same distribution, and $u_N$ denotes the $N$-th coordinate of $\mbf u$.
\end{corollary}

\medskip
\begin{remark}
The image of the function $\mbf u\mapsto g(\mbf x,\mbf u)$ is a hypersphere in $\msf T_{\mbf x}M$ with center $\tfrac12\grad f(\mbf x)$ and intersecting the origin:
Indeed, let $\mbf x\in M$ be fixed and consider the map $g_{\mbf x}:\msf S_{\mbf x}M\to\msf T_{\mbf x}M$ given by $\mbf u\mapsto g(\mbf x,\mbf u)$. To simplify formulas, we choose an orthonormal basis $(e_1,\ldots,e_N)$ in $\msf T_{\mbf x}M$ such that $\grad f(\mbf x) = \|\grad f(\mbf x)\| e_N$. In such coordinates the map $g_{\mbf x}$ is given by
$$
g_{\mbf x}(u_1,\ldots,u_N) = u_N \|\grad f(\mbf x)\| \,(u_1,\ldots,u_N).
$$
A straightforward computation shows that
$
\| \,(u_N u_1,\ldots,u_N^2-\tfrac12)\|=\tfrac12,
$
and so the image of $g_{\mbf x}$ lies on a sphere with center $\tfrac12\grad f({\mbf x})$ and passing through the origin.
This induces a probability measure on the image. Consider the function
$$
\tilde g({\mbf x},\mbf v) = \tfrac12 (\grad f({\mbf x}) + \mbf v\|\grad f({\mbf x})\|),
$$
and let {$\mbf v$} be distributed on the unit sphere according to a probability measure $\mathcal V(\msf S_xM)$ such that $g$ and $\tilde g$ induce the same probability measure. The measure $\mathcal V$ is not uniform on the sphere, but it is still invariant under rotations preserving $\grad f({\mbf x})$. We see that $v_N \sim 2u_N^2-1$.
More details on the distribution of $u_N$ and $u_N^2$ can be found in Lems.~\ref{lemma:beta-distr} and~\ref{lemma:high-dim}.
\end{remark}

Note that since the standard deviation of $u_N^2$ is larger than the expected value, Chebyshev's inequality cannot be applied to obtain useful concentration bounds. 
For large dimension $N$, there exist good approximations of these distributions. In fact, the distribution of a coordinate of a uniformly random unit vector approximately follows a normal distribution, and hence it's square approximates a $\chi^2_1$ distribution. 
See Lem.~\ref{lemma:high-dim} for a precise result.

\medskip
\begin{corollary} \label{coro:expectation}
If {$\mathbb E[\mbf u\mbf u^\top]=\mathrm{I}_N/N$},
which is satisfied for the Haar measure, then it holds that
$$
\mathbb E[g(\mbf x,\mbf u)]=\frac{\grad f(\mbf x)}{N}
$$
for any $\mbf x\in M$.
\end{corollary}

\begin{proof}
We compute: $\mathbb E[g(\mbf x,\mbf u)]=\mathbb E[\mbf u\mbf u^\top \grad f(\mbf x)] = \frac1N\grad f(\mbf x)$.
\end{proof}

The corollary shows that in this case the projected gradient is of size $\grad f(\mbf x)/N$, and hence it is rather small for large $N$. 
Intuitively, this happens because in high dimensions, a uniformly random unit vector $u$ will be close to orthogonal to the gradient with high probability --- after all, there are $N-1$ dimensions which are orthogonal to the gradient.

Now that we better understand the iteration rule, we want to understand by what amount the objective function value is likely to decrease after a certain number of iterations.
It will be useful to use normal coordinate charts at the current point $\mbf x_i$. These charts satisfy that the metric at the origin is trivial, and that geodesics passing through the origin are straight and uniformly parametrized, see for instance~\cite[Prop.~5.24]{Lee13}. In particular, when using a normal coordinate chart about $\mbf x_i$, the random variable $\mbf x_{i+1}$, conditioned on $\mbf x_i$, is still distributed on a hypersphere.

\medskip
\begin{lemma} \label{lemma:descent}
Let $(M,\msf g)$ be compact with injectivity radius $r_M$
and let $f:M\to\mathbb R$ be $\ell$-smooth. 
Then for {$\eta\leq\min(1/\ell,r_M)$} 
it holds that
$$
f(\mbf x_{i+1})-f(\mbf x_i) 
\leq 
-\eta\Big(1-\frac{\ell\eta}{2}\Big)
\langle\grad f(\mbf x_i),g(\mbf x_i,\mbf u_i)\rangle
\leq 0.
$$
In particular the function value {is surely non-increasing}.
\end{lemma}

\begin{proof}
We choose a normal coordinate chart about $\mbf x_i$ and denote the coordinates by $\mbf {\tilde x}$ and the function in coordinates by $\tilde f$. In particular $\mbf{\tilde x}_i=0$.
Note that at the origin $\grad \tilde f(\tilde{\mbf x}_i) = \nabla\tilde f(\mbf{\tilde x}_i)$ and the coordinate representation of $\mbf x_{i+1}=\exp_{\mbf x_i}(-\eta g(\mbf x_i,\mbf u_i))$ is given by $\mbf{\tilde x}_{i+1}=-\eta \tilde g(\mbf x_i,\mbf u_i)$, where $\tilde g(\mbf x_i,\mbf u_i)$ is the chart representation of $g(\mbf x_i,\mbf u_i)$.
Then, by the proof of~\cite[Lemma~1.2.3]{Nesterov04} and Lem.~\ref{lemma:g-props}, it holds that
\begin{align*}
f(\mbf x_{i+1})-f(\mbf x_i)
&=
\tilde f(\mbf{\tilde  x}_{i+1})-\tilde f(\mbf {\tilde x}_i)
\\&\leq 
\langle\nabla \tilde f(\mbf{\tilde x}_i),
\mbf{\tilde x}_{i+1}-\mbf{\tilde x}_i\rangle
+
\frac{\ell}{2}\|\mbf{\tilde x}_{i+1}-\mbf{\tilde x}_i\|^2
\\&\leq
-\eta\langle\grad f(\mbf x_i), g(\mbf x_i,\mbf u_i)\rangle
+\frac{\ell\eta^2}{2}\|g(\mbf x_i,\mbf u_i)\|^2
\\&=
-\eta\Big(1-\frac{\ell\eta}{2}\Big)\langle\grad f(\mbf x_i), g(\mbf x_i,\mbf u_i)\rangle,
\end{align*}
as desired.
\end{proof}
\noindent
This result shows that, under the assumptions of Cor.~\ref{coro:expectation}, we obtain the conditional expectation
$$
\mathbb E[f(\mbf x_{i+1})-f(\mbf x_i) \,|\, \mbf x_i]
\leq 
-\frac{\eta}{N}\Big(1-\frac{\ell\eta}{2}\Big) \|\grad f(\mbf x_i)\|^2.
$$
More generally the previous result shows that the only way for the algorithm to stop improving is for the gradient to vanish: $\grad f(\mbf x_i)\to 0$ as $i\to\infty$.

\medskip
\begin{corollary} \label{coro:conv-to-crit}
{Under the assumptions of Lem.~\ref{lemma:descent}, the function values converge (surely) to some value $c$.
Moreover,} the following assertions hold almost surely
\begin{enumerate}
\item[(i)] $c$ is a critical value of $f$.
\item[(ii)] $\mbf x_i$ converges  to the critical set of $c$, i.e.
to the set of critical point with $f(\mbf x) = c$.
\end{enumerate}
\end{corollary}

\begin{proof}
{By Lem.~\ref{lemma:descent} the function values are non-increasing and bounded below, hence they converge to some value $c$.}
We will show that $\Pr(\grad f(\mbf x_i)\nrightarrow 0)=0$, which immediately implies the result.
Since $f(\mbf x_{i+1})-f(\mbf x_i)\to0$, Lem.~\ref{lemma:descent} shows that $\langle\grad f(\mbf x_i),g(\mbf x_i,\mbf u_i)\rangle\to0$.

Note that if $\grad f(\mbf x_i)\nrightarrow 0$, then there is some $m\in\mathbb N$ such that $\|\grad f(\mbf x_i)\|>\tfrac1m$ on some infinite subsequence. 
Let $z_j^m=\langle\mbf u_{i(j)},\grad f(\mbf x_{i(j)})\rangle/\|\grad f(\mbf x_{i(j)})\|$ where $i(j)$ is the $j$-th index such that $\|\grad f(\mbf x_{i(j)})\|\geq\tfrac1m$.
Note that
$$
\langle\grad f(\mbf x_{i(j)}),g(\mbf x_{i(j)},\mbf u_{i(j)})\rangle
=
(z_j^m\|\grad f(\mbf x_{i(j)})\|)^2\,.
$$
Hence, if the gradients do not converge to $0$, then, for some $m$, it must hold that $|z_j^m|\to0$.

We claim that there exist $0<\varepsilon,\delta<1$ such that for all $\mbf x\in M$ with $f(\mbf x)\leq f(\mbf x_0)$ it holds that 
$$
\Pr(|\langle\mbf u,\grad f(\mbf x)\rangle|\leq\varepsilon\|\grad f(\mbf x)\|)\leq\delta\,.
$$
If $\mbf u\sim\mathcal U(\mathsf S_{\mbf x}M)$ this follows from Lem.~\ref{lemma:beta-distr}, if $\mbf u\sim\mathcal D(\mathsf S_{\mbf x}M)$ this follows from Lem.~\ref{lemma:overlap}.

Given $m,n,k\in\mathbb N$ with $1/n<\varepsilon$, and using the chain rule, Markovianity, and the claim above, we compute that
\begin{align*}
\Pr(\forall j>k:|z_j^m|\leq\tfrac1n) 
=\prod_{j=k}^\infty \Pr(|z_j^m|\leq\tfrac1n \;\big|\;\forall l\in\{j,\ldots,k\}:|z_l^m|\leq\tfrac1n) 
=0\,.
\end{align*}
For any fixed value of $n$, we can use a union bound over all $m,k\geq1$ and obtain
$$
\Pr(\exists m,k\,\forall j>k:|z_j^m|\leq\tfrac1n)=0\,,
$$
which implies that 
$$
\Pr(\exists m : |z_j^m|\to0 \text{ as }j\to\infty)=0\,,
$$
as desired.
\end{proof}

\subsection{Escaping Saddle Points} \label{sec:escape-saddles}

The main difficulty in proving almost sure convergence to a local minimum of the randomly projected gradient descent algorithm is to show that it does not get stuck in a saddle point.

Recall from Lem.~\ref{lemma:descent} that the function value cannot increase. 
Hence, once we cross the critical value corresponding to some saddle point, we say that we have \emph{passed} the saddle point, as now it is impossible to converge to the saddle point in question.

We start by proving the result in a simplified ``isotropic'' case, before treating the general case as a perturbation.

\medskip
\begin{lemma} \label{lemma:theta-proba}
Consider the function $f:\mathbb R^N\to\mathbb R$ defined by 
{
$$
f(x,y_,z)=a_1x_1^2+\ldots+a_px_p^2-(b_1y_1^2+\ldots+b_qy_q^2),
$$
with $x \in \mathbb R^p$, $y \in \mathbb R^q$,  $z \in \mathbb R^{N-p-q}$,
and $p,q\geq1$} as well as $a_i,b_j>0$. Then $f$ is $\ell$-smooth with 
{$\ell=2 \max\{a_1,\dots, a_p,b_1, \dots, b_q\}$}.
{Moreover, let $0<\eta\leq\tfrac{1}{\ell}$ and define
\begin{equation} \label{eq:saddle-angle}
\theta := \arctan\Big(\frac{b_1y_1^2 + \ldots + b_qy_q^2}{a_1x_1^2 + \ldots + a_px_p^2}\Big) \in \big[0,\tfrac\pi2\big]
\quad \text{for} \quad (x_1,\dots,x_p,y_1, \dots, y_q) \neq 0.
\end{equation}}
Then there exist constants $\varepsilon,\delta>0$ such that for all%
\,\footnote{The value $\tfrac\pi3$ is arbitrary and could be replaced by any other value in $(\tfrac\pi4,\frac\pi2)$ without changing the proof.}
$\theta\in[0,\tfrac\pi3]$ it holds that 
$$
\Pr(\theta_{i+1}-\theta_{i}\geq\varepsilon)\geq\delta.
$$
\end{lemma}

\begin{proof} 
The negative gradient of $f$ is 
\begin{align*}
-\nabla f(x_1,\ldots,x_p,y_1,&\ldots,y_q,z_1,\ldots,z_{N-p-q})
\\&=-2(a_1x_1,\ldots,a_px_p,-b_1y_1,\ldots,-b_qy_q, 0,\ldots,0)
\end{align*}
and the value of $\ell$ follows immediately.
A key property of this simplified setting is that the gradient is a linear vector field on $\mathbb R^N$, and so the entire situation is invariant under scaling.
{Moreover, from the definition of $\theta$ it is clear that
$f$ vanishes if and only if $\theta=\tfrac\pi4$ or 
$(x_1,\dots,x_p,y_1, \dots, y_q)=0$. Due to the scaling} invariance of the algorithm and of~\eqref{eq:saddle-angle}, we may focus on an initial point {$\mbf x_i = (x_{i},y_{i},z_{i})$} 
on the unit sphere $S^{N-1}$.
After executing one step, we obtain a random improvement of our angle, namely $\theta_{i+1}-\theta_i$ which depends on $\mbf u_i$ appearing in the update rule. 

{First let us consider this improvement in the uniform Case~\ref{it:cont}}
as a function of $\mbf x_i$ and $\mbf u_i$, i.e.,
$$
h(\mbf x_i,\mbf u_i) := \theta_{i+1}-\theta_i\,.
$$
{It is} clear that {this} is a continuous function. 
The optimal improvement achievable for a given $\mbf x_i$ is found by maximizing over $\mbf u_i$. 
One can show that the resulting function is still continuous in $\mbf x_i$.
Moreover we claim that it is strictly positive 
{on the compact set $K_{\pi/3} := \{\mbf x\in S^{n-1} : \theta(\mbf x)\leq\pi/3\}$.}
Indeed, if $\theta_i>0$, then this follows from the choice of $\eta$. 
If we perform a single (deterministic) negative gradient step with step size $\eta$, none of the coordinates will change sign, and whenever {$\theta_i\in(0,\tfrac\pi3]$} its value will strictly increase.
If $\theta_i=0$, it is clear that the value of {$\theta_i$} will almost surely strictly increase, as the set of points $\mbf x_{i+1}$ with $\theta_{i+1}=0$ form a strict subspace of $\mathbb R^n$.
Taken together this shows that there is some $\varepsilon>0$ such that the optimal improvement is at least $2\varepsilon$ for every initial {$\mbf x_{i}
\in K_{\pi/3}$.}

It follows that for any given initial state {$\mbf x_{i} \in K_{\pi/3}$}, the probability that $\theta_{i+1}-\theta_{i}\geq\varepsilon$ is strictly positive. 
Indeed, {$h(\mbf x_i,\cdot)$ is a continuous function on this sphere and its preimage of $(\varepsilon,+\infty)$ on the sphere is a non-empty open set and hence has a strictly positive probability (with respect to the Haar measure).}
Moreover one can show that this probability depends continuously on the initial state $\mbf x_{i}$, and again we obtain a positive minimum $\delta$ on the {$K_\theta$} by compactness.
This proves that $\Pr(\theta_{i+1}-\theta_{i}\geq\varepsilon)\geq\delta$ 
independently of the initial state {$\mbf x \in K_{\pi/3}$}.

{Next} consider the discrete Case~\ref{it:fin}.
Again we need to prove that there exist constants $\varepsilon,\delta>0$ such that for all $\theta\in[0,\tfrac\pi3]$ it holds that $\Pr(\theta_{i+1}-\theta_{i}\geq\varepsilon)\geq\delta$.
The only cases where $h(\mbf x,\mbf u)=0$ (note that it cannot be negative) with probability one is when all $v_j(\mbf x)$ for $j=1,\ldots,k$ are either orthogonal to $\grad(\mbf x)$ or collinear to $\mbf x$. 
By Assumption~\ref{assump}, this cannot happen.
Hence $\max_j h(\mbf x,\mbf v_j)$ is non-negative and continuous, and on the compact set {$K_{\pi/3}$} it is even strictly positive. 
This concludes the proof.
\end{proof}

As a corollary, we are able to show that after sufficiently many steps, the algorithm passes the saddle point with high probability.

\medskip
\begin{corollary} \label{coro:saddle-as}
In the same setting as above, there exist $\varepsilon,\delta>0$ such that, with $n=\lceil\tfrac{\pi}{3\varepsilon}\rceil$, we have that
$$
\Pr(\theta_n\geq\tfrac\pi3) 
\geq \delta^n.
$$
It follows that for any $m\in\mathbb N$, it holds that
$$
\Pr(\theta_{mn}<\tfrac\pi3)
\leq 
\Pr(\theta_n<\tfrac\pi3)^m
\leq 
(1-\delta^n)^m,
$$
which goes to $0$ as $m\to\infty$.
\end{corollary}

\begin{proof}
It follows immediately from Lem.~\ref{lemma:theta-proba} that in $n$ steps we achieve $\Pr(\theta_{i+n}-\theta_i\geq n\varepsilon)\geq\delta^n$, or put differently, setting $n=\lceil\tfrac{\pi}{3\varepsilon}\rceil$, we get
$$
\Pr(\theta_n\geq\tfrac\pi3) 
\geq
\Pr(\theta_n\geq n\varepsilon)
\geq 
\delta^n
=
\delta^{\lceil\tfrac{\pi}{3\varepsilon}\rceil}
>
0.
$$
This proves the first statement. The second statement follows from the Markovianity of the process.
\end{proof}

In order to generalize Cor.~\ref{coro:saddle-as} to the desired setting we will treat the effects of the Riemannian metric and the higher order terms of $f$ as a perturbation on the algorithm.
We start with the following general perturbation result.

\medskip
\begin{lemma} \label{lemma:perturbation}
Let $G$ be a linear vector field on $\mathbb R^n$ and let $\eta>0$ be given. Moreover let $P$ be {any other}
vector field on $\mathbb R^n$ satisfying $\|P({\mbf x})\|\leq 
C\|{\mbf x}\|^2$ for some $C>0$.
Then it holds that 
$$
\|{\mbf x}-\eta G{\mbf x} - ({\mbf y}-\eta (G+P){\mbf y})\|
\leq 
\|\id+\eta G\| \|{\mbf x}-{\mbf y}\| + C\eta\|{\mbf y}\|^2,
$$
where $\|\id+\eta G\|$ denotes the operator norm.
Let ${\mbf x}_{i+1}={\mbf x}_i-\eta G{\mbf x}_i$ and ${{\mbf y}_{i+1}}={\mbf y}_i-\eta (G+P){\mbf y}_i$.
Hence after $n$ steps, if the states remain in an $R$-ball about the origin, then
$$
\|{\mbf x}_{i+n}-{\mbf y}_{i+n}\|
\leq 
\|\id+\eta G\|^n \Big(\|{\mbf x}_{i}-{\mbf y}_{i}\| + \frac{CR^2}{\|\id+\eta G\|-1}\Big).
$$
\end{lemma}

\begin{proof}
The first inequality follows from the assumptions using the triangle inequality:
\begin{align*}
\|{\mbf x}-\eta G{\mbf x} - ({\mbf y}-\eta (G+P){\mbf y})\|
&=
\|(\id+\eta G)({\mbf x}-{\mbf y}) + \eta P({\mbf y})\|
\\&\leq
\|\id+\eta G\| \|{\mbf x}-{\mbf y}\| + C\eta\|{\mbf y}\|^2.
\end{align*}
Now using the assumption that $\|{\mbf y}\|\leq R$, and by iterating this result for $n$ steps, we find that
\begin{align*}
\|{\mbf x}_{i+n}-{\mbf y}_{i+n}\|
&\leq 
\|\id+\eta G\|^n\|{\mbf x}_{i}-{\mbf y}_{i}\| + (\|\id+\eta G\|^{n-1} + \cdots + 1)CR^2
\\&\leq 
\|\id+\eta G\|^n \Big(\|{\mbf x}_{i}-{\mbf y}_{i}\| + \frac{CR^2}{\|\id+\eta G\|-1}\Big).
\end{align*}
This concludes the proof.
\end{proof}

In our case, we can always choose coordinates in which the problem locally looks like the setting of Lem.~\ref{lemma:theta-proba} with a perturbation as in Lem.~\ref{lemma:perturbation}. 

\medskip
\begin{lemma}
\label{lemma:local-form}
Let $(M,\mathrm g)$ be a Riemannian manifold and let $f:M\to\mathbb R$ be a Morse--Bott function on $M$.
If $\mbf x_0\in M$ is a critical point of $f$, then there exists a chart $\sigma:U\to\mathbb R^n$ about $\mbf x_0$ such that, with $\tilde{\mbf x}=\sigma(\mbf x)=(x_1,\ldots,x_p,y_1,\ldots,y_q,z_1,\ldots,z_{N-p-q})$,
\begin{align*}
\sigma(\mbf x_0)&=0, \\
f\circ\sigma^{-1}(\tilde{\mbf x}) &= c + a_1x_1^2+\ldots+a_px_p^2-(b_1y_1^2+\ldots+b_qy_q^2) + \mathcal O(\|\tilde{\mbf x}\|^2), \\
\mathrm  g_{ij}(\tilde{\mbf x}) &= \delta_{ij} + \mathcal O(\|\tilde{\mbf x}\|),
\end{align*}
where $\delta_{ij}$ denotes the Kronecker symbol.
\end{lemma}

\begin{proof}
Choose coordinates about $\mbf x_0$ such that the metric is Euclidean at $\tilde{\mbf x}_0$ (e.g.\ normal coordinates), then use an orthogonal transformation to diagonalize the Hessian of $f$ at $\tilde{\mbf x}_0$.
\end{proof}

\noindent Now we can combine the three previous results.

\medskip
\begin{corollary} \label{coro:local-escape}
Let $M$ be a Riemannian manifold and let $f$ be a Morse--Bott function on $M$.
For any strict saddle point ${\mbf x}\in M$ with value $f({\mbf x})=c$ there exists {$N>0$} and a neighborhood $U$ of ${\mbf x}$ such that 
$$
{\Pr(f({\mbf x}_{i+N}) < c \,|\, {\mbf x}_{i},\ldots,{\mbf x}_{i+N}\in U)}
>\delta.
$$
\end{corollary}

\begin{proof}
This follows immediately from Cor.~\ref{coro:saddle-as}, Lem.~\ref{lemma:perturbation} and Lem.~\ref{lemma:local-form}.
\end{proof}

\medskip
\begin{lemma}
\label{lemma:saddle-as-gen}
Let $M$ be a Riemannian manifold and let $f$ be a Morse--Bott function on $M$.
Further assume that $f:M\to\mathbb R$ has compact sublevel sets and that $\mbf x_0\in M$ is not a critical point of $f$.
Then the probability that $\mbf x_{i}$ converges to 
{the set of strict saddle points}
as $i\to\infty$ is zero.
\end{lemma}

\begin{proof}
{Let $C$ be the set of all strict saddle points $\mbf x$ with 
$f(\mbf x) \leq f(\mbf x_0)$. By assumption on $f$ the set $C$ is compact; actually, $C$ 
consists of finitely many connected compact submanifolds, cf.~\cite[Def.~2.41]{Nicolaescu11}.
We will show that the probability that $\mbf x_i$ converges to $C$ is zero. Note that by Lem.~\ref{lemma:descent} the function 
value is (surely) not increasing, and therefore we can focus on $C_c := C \cap f^{-1}$$(c)$ for some $c \leq f(\mbf x_0)$. By compactness of $C_c$ we can choose for each $\mbf z \in C_c$ balls $B_{r_{z}}(\mbf z)$, $B_{R_{z}}(\mbf z)$ around $\mbf z$ with $r_z < R_z$ such that both balls are  contained in a neighborhood $U(\mbf z)$ of $\mbf z$ as in Cor.~\ref{coro:local-escape}  and such that any realization $\mbf x_i$ needs at least $N$ steps to transverse
the spherical shells $B_{R_z}(\mbf z) \setminus B_{r_z}(\mbf z)$. Since $C_c$ is
compact we can find a finite covering by balls of the form $B_{r_z/2}(\mbf z)$, 
i.e.
$$ C_c \subset \bigcup_{k=1}^K B_{r_k/2}(\mbf z_k)\,,$$
Next define $r_{*} := \min_{k = 1, \dots K} r_k/2$. Then, by assumption,  there exists 
some $L \in \mathbb N$ such that $d(\mbf x_i,C_c) < r_*$ for all 
$i \geq L$. Hence we conclude $d(\mbf x_L, \mbf z) \leq r_*$ for some $\mbf z \in C_c$ 
and moreover $\mbf z \in B_{r_k/2}(\mbf z_k)$ for some $k \in \{1, \dots, K\}$. Thus,
the triangle inequality implies $\mbf x_L \in B_{r_k}(\mbf z_k)$ and therefore we conclude
$\mbf x_L, \mbf x_{L+1}, \dots, \mbf x_{L+N}$ is in $U(\mbf z_k)$. Since $\mbf x_{L+N}$
again satisfies $d(\mbf x_{L+N},C_c) < r_*$ we can repeat the above arguments and obtain
$\mbf x_{L+N}, \mbf x_{L+N+1}, \dots, \mbf x_{L+2N}$ are in some $U(\mbf z_{k'})$, etc.
Combining this observation  with Cor.~\ref{coro:local-escape} we get
\begin{equation*}
\Pr(\mbf x_i \to C_c \;\text{for}\; i \to \infty)
\leq \sum_{L = 1}^\infty {\Pr}_L = 0\,,
\end{equation*}
where $\Pr_L$ is defined by 
\begin{equation*}
{\Pr}_{L} :=
\Pr \left(
\begin{minipage}{75mm}
$ \forall i \geq L: \, f(\mbf x_i) \geq c \;\wedge\; \forall m  \in \mathbb N \, 
\exists k \in \{1, \dots K\}: \mbf x_{L+(m-1)N}, \dots, \mbf x_{L+mN} \in U(\mbf z_k)$
\end{minipage}
\right) = 0\,.
\end{equation*}
Above we used the fact that by Markovianity the behavior of each sequence of 
$N$ consecutive steps behaves independently. Since there are only finitely  many critical submanifolds contained in $C$ it follows that there are only finitely many $c \leq f(\mbf x_0)$ such that $C_c \neq \emptyset$ and thus we conclude that the probability that $\mbf x_i$ converges to $C$ is also zero. }
\end{proof}

\subsection{Almost Sure Convergence} \label{sec:almost-sure-convergence}

The results proven so far guarantee that the algorithm {(cp.\ Fig.~\ref{fig:intro1})} converges almost surely to the set of local minima:

\medskip
\begin{proposition} 
\label{prop:as-conv}
Let $M$ be a Riemannian manifold and let $f:M\to\mathbb R$ be an $\ell$-smooth Morse--Bott function with compact sublevel sets.
Assume that $\mbf x_0$ is not a critical point.
Then, for stepsize $\eta\leq1/\ell$, the randomized gradient descent algorithm converges almost surely to the set of local minima.
\end{proposition}

\begin{proof}
Since by Lem.~\ref{lemma:descent} the function value cannot increase, and since by compactness of sublevel sets the function is lower bounded, the function value must (surely) converge to some value $f^\star$. 
By Cor.~\ref{coro:conv-to-crit}, $f^\star$ is almost surely a critical value. 
By Lem.~\ref{lemma:saddle-as-gen} the algorithm almost surely does not converge towards {the set of strict saddle points.  Hence it has converge almost 
surely to the set of local minima.}
\end{proof}

In order to prove convergence to a single local minimum, we first need the following technical lemma.

\medskip
\begin{lemma} \label{lemma:exponential-as}
Let $\mathcal E_i$ be probability distributions over the interval $[0,1]$ and let $E_i\sim\mathcal E_i$ for $i\in\mathbb N$. 
Let $\varepsilon>0$ be given. 
We assume that there is $q>0$ such that
$$\Pr(E_i\geq\varepsilon\,|\,E_1,\ldots,E_{i-1})\geq q$$ 
for all $i\in\mathbb N$.
Then for any $\alpha\in(0,q\ln\tfrac{1}{1-\varepsilon})$ it holds that almost surely
$$
\prod_{i=1}^n (1-E_i) \leq e^{-\alpha n}
$$
for $n$ large enough.
\end{lemma}

\begin{proof}
First recall the following basic fact. 
Consider a sequence of i.i.d. biased coin tosses $X_i\in\{0,1\}$ with $\Pr(X_i=1)=q$.
By the strong law of large numbers~\cite[Thm.~20.1]{Jacod04} the average $\overline X_n=\tfrac1n\sum_{i=1}^n X_i$ converges almost surely to the expectation $q$. 
Hence for any $p\in (0,q)$, there almost surely exists some $n_0$ large enough such that $\overline X_n\geq p$ for all $n>n_0$.

Let any value $n$, and index set $I=\{i_1,\ldots,i_k\}\subseteq\{1,\ldots,n\}$ (ordered increasingly) be given.
Then, using the law of total probability,
\begin{align*}
\Pr(E_i\geq\varepsilon, \,\,\forall i\in I)
&=
\Pr(E_{i_k}\geq\varepsilon\wedge\ldots\wedge E_{i_1}\geq\varepsilon)
\\&=
\Pr(E_{i_k}\geq\varepsilon \,|\,
E_{i_{k-1}}\geq\varepsilon\wedge\ldots\wedge E_{i_1}\geq\varepsilon) 
\ldots 
\Pr(E_{i_1}\geq\varepsilon)
\\&\geq 
q^k
\\&=
\Pr(X_{i_k}=1,\ldots X_{i_1}=1)
\end{align*}
and thus we have the following:
For all $p\in (0,q)$, there almost surely exists some $n_0$ large enough such that for all $n>n_0$ it holds that at least $pn$ of the $E_i$ are greater than or equal to $\varepsilon$.
In this case
$$
\prod_{i=1}^n (1-E_i) 
\leq 
(1-\varepsilon)^{np}
=e^{-\alpha n},
$$
with $\alpha=p\ln\tfrac{1}{1-\varepsilon}$.
\end{proof}

Finally this allows us to prove the main result:

\medskip
\begin{theorem} 
\label{thm:almost-sure-convergence}
Let $M$ be a Riemannian manifold and let $f:M\to\mathbb R$ be an $\ell$-smooth Morse--Bott function with compact sublevel sets.
Further assume that ${\mbf x}_0$ is not a critical point of $f$.
Then, for step size $\eta\leq1/\ell$, the randomly projected gradient descent algorithm converges almost surely to a local minimum. 
\end{theorem}

\begin{proof}
The idea {of} the proof is as follows: first we show that almost surely the distance to the set of local minima decreases exponentially. {Then this} implies that the size of each step decreases exponentially as well, and hence the algorithm converges absolutely to a local minimum.

Prop.~\ref{prop:as-conv} shows that ${\mbf x}_i$ converges almost surely to some set of local minima denoted $C$.
Let ${\mbf z}\in C$ be such a local minimum, and let $H({\mbf z})$ denote the Hessian at ${\mbf z}$, and let $a_{\min}({\mbf z})$ denote the smallest non-zero eigenvalue of $H({\mbf z})$.
{Assume w.l.o.g.~that $f(\mbf z)=0$.}
By compactness and continuity $a_{\min}({\mbf z})$ has a minimum on $C$, which we denote $a_{\min}$.
{In a sufficiently small neighborhood of ${\mbf z}$, working in a chart as provided by Lem.~\ref{lemma:local-form}, we obtain}
\begin{align*}
\|\grad f({\mbf x})\|^2
&=
\sum_j 4 a_j(\mbf z)^2 (x_j-z_j)^2 + \mathcal O(\|{\mbf x}-\mbf z\|^3)
\\&\geq
4 a_{\min}(\mbf z) f({\mbf x}) + \mathcal O(\|{\mbf x}-\mbf z\|^3)
\\&\geq
2 a_{\min}(\mbf z) f({\mbf x}).
\end{align*}
Together with Lem.~\ref{lemma:descent} and Cor.~\ref{coro:g-props}
we find
$$
f({\mbf x}_{i+1}) \leq \Big(1-((\mbf u_i)_N)^2 \delta 2a_{\min}\Big)f({\mbf x}_{i})
$$
{with $\delta=\eta(1-\tfrac{\ell\eta}{2}) < \tfrac{1}{2 a_{\min}}$.
By Lem.~\ref{lemma:exponential-as}, applied to
$E_i = ((\mbf u_i)_N)^2 \delta a_{\min}$}, this shows that there is some
$\alpha>0$ such that $f({\mbf x}_{i})\leq f({\mbf x}_{0}) e^{-\alpha i}$ 
almost surely for $i$ large enough. Again by Lem.~\ref{lemma:descent} we have that 
$$
\|{\mbf x}_{i+1}-{\mbf x}_{i}\|
\leq 
\sqrt{\frac{|f({\mbf x}_{i+1})-f({\mbf x}_{i})|}{\tfrac{1}{\eta}-\tfrac{\ell}{2}}}.
$$
This shows that {for $L \in \mathbb N$ large enough we can assume
that all $\mbf x_i$ with $i \geq L$ are contained in a single chart of 
the same kind as in Lem.~\ref{lemma:descent} and thus we conclude}
$$
\sum_{i=L}^\infty \|{\mbf x}_{i+1}-{\mbf x}_{i}\| 
\leq 
\frac{1}{\sqrt{\tfrac{1}{\eta}-\tfrac{\ell}{2}}}
\sum_{i=L}^\infty \sqrt{f({\mbf x}_{i})}
\leq
\sqrt{\frac{f({\mbf x}_{0})}{\tfrac{1}{\eta}-\tfrac{\ell}{2}}}
\frac{e^{L \alpha/2}}{1-e^{\alpha/2}}
< \infty,
$$
and so the total length of the trajectory is finite.
Thus ${\mbf x}_{i}$ almost surely {converges} to a local minimum of $f$.
\end{proof}

\section{Quantitative Results in Low Dimensions} \label{sec:low-dims}

The goal of this section is to study how long it takes to pass a saddle point.
More precisely, we want to understand the hitting time $\tau$ of the sublevel set of the critical value at a saddle point. 
By construction, the discrete stochastic process ${\mbf x}_i$ is a Markov chain, since the step ${\mbf x}_{i+1}-{\mbf x}_i$ depends only on the value of ${\mbf x}_i$, and not on previous steps.
As such, for small enough step size $\eta$, we can approximate the algorithm by an appropriate Ito stochastic differential equation (SDE). As we will see, the accuracy of this approximation improves with smaller step size $\eta$, at the cost of slower convergence.

\subsection{Two-Dimensional Euclidean Case} \label{sec:numerics}

In the simplest case, we consider the prototypical Morse function $f(x,y)=x^2-y^2$ on $\mathbb R^2$ in the Euclidean metric.
Due to the scale invariance, we can normalize the state after each iteration and so we obtain a discrete-time stochastic process on the unit circle parametrized by the angle $\phi$.
For small enough step size $0<\eta\ll1$ the iteration rule is approximately%
\footnote{The formula is obtained by locally approximating the unit circle by its tangent line. More precisely, we define $\Delta\phi=-\eta\langle\mbf u,\nabla f\rangle\langle\mbf u,\mbf v\rangle\mbf v$ where $\mbf v=(-\sin\phi,\cos\phi)$. Then $\Delta\phi=\eta\sin(2\phi)-\eta2u_1u_2\overset{d}{=}\eta\sin(2\phi)+\eta u_2$.} 
given by
\begin{align} \label{eq:sde-discrete}
\Delta\phi_i=\phi_{i+1}-\phi_i 
\overset{d}=
\eta\,(\sin(2\phi_i)+u_2) \,,
\end{align}
where $\mbf u=(u_1,u_2)$ is chosen uniformly at random from the unit circle and hence using Lem.~\ref{lemma:beta-distr} we find that
$$
\mathbb E[\Delta\phi_i]
=
\eta\,\sin(2\phi_i),
\quad 
\operatorname{Var}(\Delta\phi_i)
=
\frac\eta2\,.
$$
Using~\cite[Sec.~6.2]{Kloeden92} we can approximate this process using the stochastic differential equation (SDE)
\begin{align} \label{eq:sde-cont}
d\Phi_t = \sin(2\Phi_t) dt
+
\frac{\sqrt\eta}2 dW_t\,,
\end{align}
where $W_t$ denotes a standard Brownian motion.
The approximation is visualized in Fig.~\ref{fig:sde-approx}.

\begin{figure}[htb]
\centering
\includegraphics[width=0.45\textwidth]{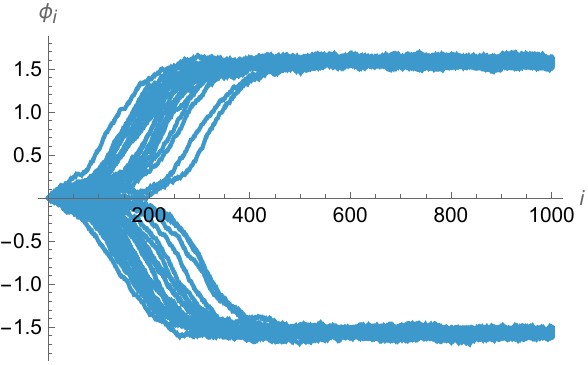}
\quad
\includegraphics[width=0.45\textwidth]{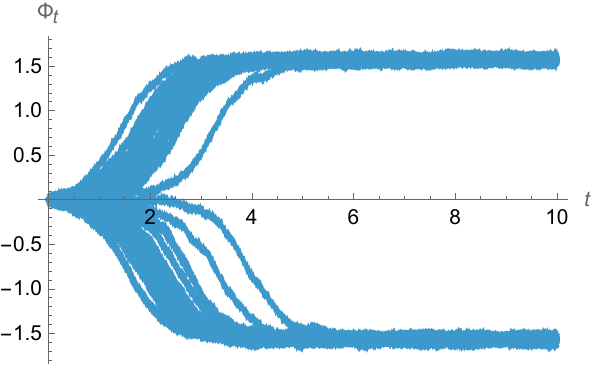}
\caption{We use the step size $\eta=0.01$, total time $T=10$, and initial state $\phi_0=\Phi_0=0$.
Left: $50$ realizations of the discrete-time stochastic process~\eqref{eq:sde-discrete} for $T/\eta=1000$ steps.
Right: $50$ realizations of the continuous-time stochastic process~\eqref{eq:sde-cont} for time $T=10$, simulated using the Euler--Maruyama scheme~\cite[Sec.~9.1]{Kloeden92} with time step size $\Delta t=0.001$.}
\label{fig:sde-approx}
\end{figure}

In order to understand how long it takes the algorithm to pass the saddle point, we are interested in computing the hitting time 
$$
\tau = \inf\{t\geq0 : \Phi_t=\pm\tfrac\pi4\},
$$
where we always assume that $\Phi_0=0$.
Close to $0$, the SDE can be linearized as $d\Phi_t=2\Phi_t dt + \tfrac{\sqrt\eta}2 dW_t$, which is a mean-repelling Ornstein--Uhlenbeck process. 
Away from $0$, the deterministic part dominates and we can solve the (deterministic) ODE $\dot\Phi_t = \sin(2\Phi_t)$.

First we approximate the hitting time distribution of the mean-repelling Ornstein--Uhlenbeck process. More precisely, we find a lower bound on the cumulative distribution function (c.d.f.) of $\tau_c$, the hitting time of $\pm c$.

\medskip
\begin{lemma} \label{lemma:approx-ornstein}
Let $X_t$ be the solution of the SDE $dX_t=\kappa X_t dt + \sigma dW_t$ with $\kappa,\sigma>0$ and $X_0=0$. 
Setting $\tilde\sigma(t)=\sigma\sqrt{\frac{e^{2\kappa t}-1}{2\kappa}}$, it holds that $X_t\sim\mathcal N(0,\tilde\sigma(t)^2)$. 
If we denote by $\tau_c$ the hitting time of $\pm c$ (where $c>0$), we find the lower bound $\Pr[\tau_c\leq t]\geq\Pr[|X_t|\geq c]=1+\operatorname{erf}\big(\frac{-c}{\tilde\sigma(t)\sqrt2}\big)$ where $\operatorname{erf}(\cdot)$ denotes the error function.
\end{lemma}

\begin{proof}
The SDE is linear and hence has the well-known solution
$$
X_t 
= 
\sigma \int_0^t e^{\kappa(s-t)}dW_s 
\sim 
\mathcal N(0,\tilde\sigma(t)^2),
$$
where $\tilde\sigma(t)=\sigma\sqrt{\frac{e^{2\kappa t}-1}{2\kappa}}$, see~\cite[Sec.~4.2 \&~4.4]{Kloeden92}.
It is clear that if $|X_t|\geq c$ then $\tau_c\leq t$, which implies the last statement.
\end{proof}

The approximation is very accurate, as illustrated in Fig.~\ref{fig:approx-ornstein}. 
The following remark gives some details on the quality of the approximation.

\begin{figure}[htb]
\centering
\includegraphics[width=0.45\textwidth]{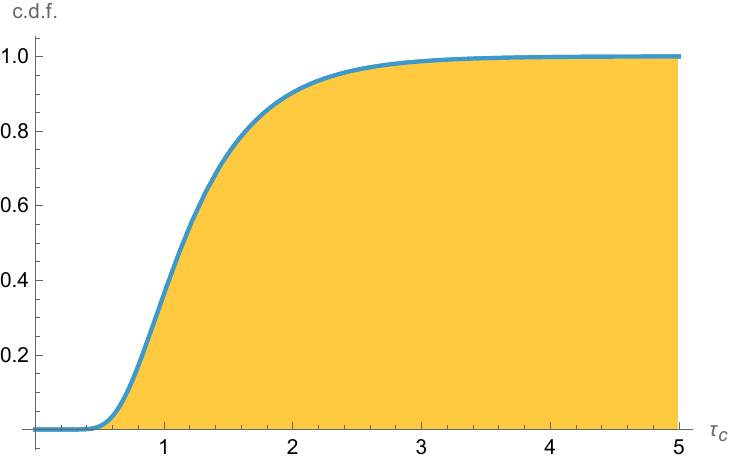}
\quad
\includegraphics[width=0.45\textwidth]{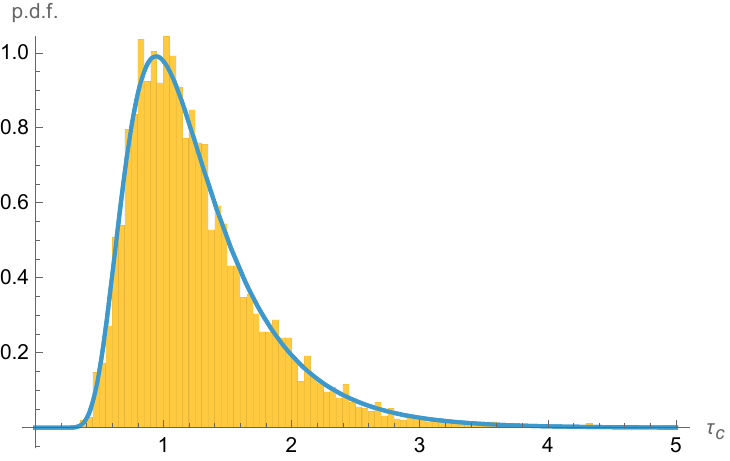}
\caption{Lower bound for the hitting time $\tau_c$ of the mean-repelling Ornstein--Uhlenbeck process, as obtained in Lem.~\ref{lemma:approx-ornstein}.
We used the constants $\kappa=2$, $\sigma=3$ and $c=10$. 
Shown are the histograms of the cumulative distribution function (c.d.f.) and probability density function (p.d.f.) of $\tau_c$ for the SDE $dX_t=\kappa X_t dt + \sigma dW_t$ computed using time steps of size $0.001$ (orange), and the analytical lower bound for the c.d.f.\ and the resulting p.d.f.\ obtained in Lem.~\ref{lemma:approx-ornstein} (blue).}
\label{fig:approx-ornstein}
\end{figure}

\medskip
\begin{remark}
As in the proof of Lem.~\ref{lemma:approx-ornstein} the solution of the SDE with initial condition $X_0=c$ is given by
$$
X_t = e^{\kappa t}c + \sigma \int_0^t e^{\kappa(s-t)}dW_s. 
$$
Then 
$$
e^{\kappa t}c - n\tilde\sigma(t) \geq c
\implies
c \geq n\sigma \sqrt{\frac{e^{2\kappa t}-1}{2\kappa}}\frac{1}{e^{\kappa t}-1}
\geq \frac{n\sigma}{\sqrt{2\kappa}}\max\Big(1,\sqrt{\tfrac2{\kappa t}}\Big).
$$
Hence, for the error in $\tau_c$ to be small compared to some given $t>0$, we need $c\gg\frac{\sigma}{\sqrt\kappa}$ if $\kappa t$ is big, and $c\gg \frac\sigma{\kappa\sqrt{t}}$ if $\kappa t$ is small.
\end{remark}

\smallskip
Now if we choose $c$ small enough that the linearization of $\sin(\cdot)$ is accurate, but large enough that the approximation of Lem.~\ref{lemma:approx-ornstein} is good (which is always possible if $\eta$ is small enough), then we can approximate $\tau$ as follows. 

We fix some value $c$ and note that the ODE $\dot\phi(t)=\sin(2\phi(t))$ with initial condition $\phi(0)=c$ has the solution $\phi(t)=\arctan(e^{2t}\tan(c))$, and hence the hitting time of $\pi/4$ is $\tilde\tau_c=-\tfrac12\ln(\tan(c))$. Hence we have approximately $\tau\simeq\tau_c+\tilde\tau_c$. This is illustrated in Fig.~\ref{fig:hitting-times} for realistic values.

\begin{figure}[htb]
\centering
\includegraphics[width=0.5\textwidth]{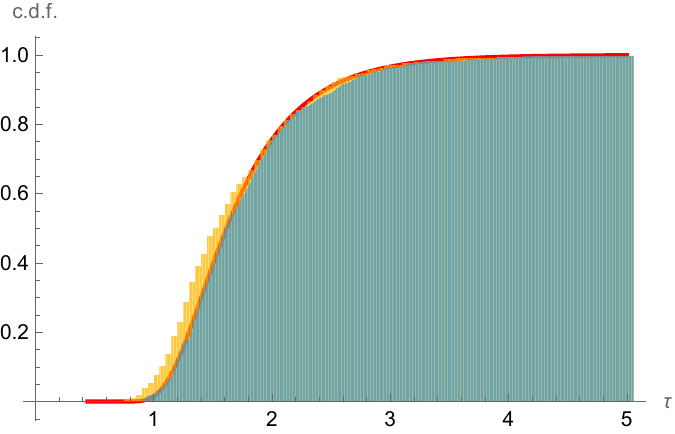}
\caption{Using $\eta=0.05$ we plot the c.d.f. of the hitting time $\tau$ of $\tfrac\pi4$. 
In orange for the difference equation~\eqref{eq:sde-discrete}, in blue for the stochastic differential equation~\eqref{eq:sde-cont} (with time step $\Delta t=0.001$), and in red for the analytic approximation as described above.}
\label{fig:hitting-times}
\end{figure}

\section{{Major Application:\\ Ground-State Optimization in Quantum Physics}} \label{sec:gs-opt}

Among the Riemannian optimization
techniques adopted in quantum information science,
hybrid quantum-classical algorithms \cite{bharti2022noisy} such as variational quantum algorithms (VQAs) \cite{cerezo2021variational, tilly2022variational}, have been developed to solve ground-state (smallest eigenvalue) problems \cite{peruzzo2014variational}.

From the perspective of quantum information and computation,
in VQAs, a fixed parameterized quantum circuit, i.e.~a parameterized set of unitary transformations (usually implemented by a fixed set of unitary gates), is iteratively optimized in tandem with a classical computer to minimize a cost function whose global optima encode the solution(s) to the desired problem. 
More precisely, a quantum computer, including a measurement device, is used to output the current (e.g., expectation) value of the cost function in each iterative step. 
Interleaved, a classical computer component provides the optimization routine that determines the parameter update for the next iteration on the quantum computer. 
Since quantum computers can, in principle, handle some classically exponentially complex problems, the hope is that VQAs will be able to solve problems of large size that a solely classical algorithm would be incapable of. 

However, since the optimization problems VQAs aim to solve are highly non-linear, which can often be traced back to the quantum circuit parametrization \cite{magann2021pulses, lee2021progress}, the iterative search for the optimal parameters can get stuck in suboptimal solutions \cite{bittel2021training}. 
To overcome this challenge, adaptive quantum algorithms have been designed.  
Instead of fixing a quantum circuit and picking a parametrization as in VQAs, adaptive algorithms successively grow a quantum circuit based on measurement data from the quantum computer. 
In this approach, the cost function is minimized \textit{while} the quantum circuit is grown, see \cite{grimsley2019adaptive,wiersema2023optimizing, magann2022feedback, magann2023randomized}. 
One of the most prominent examples of such a strategy is the so-called {\sc adapt-vqe} algorithm \cite{grimsley2019adaptive}, which was originally proposed to solve ground state problems in chemistry, and was further developed later to solve combinatorial optimization problems on quantum computers \cite{tang2021qubit}. 

Unfortunately, since typically in adaptive quantum algorithms the circuit growth is informed by gradient estimates performed by the quantum computer, adaptive quantum algorithms face similar challenges as VQAs, i.e.\ the adaptive circuit growth can get stuck when gradients (asymptotically) vanish.
This issue has recently been addressed by identifying adaptive quantum algorithms as quantum-computer implementations of Riemannian gradient flows on the special unitary group $\text{SU}(d)$ of dimension $d=2^{n}$ where $n$ is the number of qubits of the quantum computer \cite{wiersema2023optimizing}. 
Since implementing Riemannian gradient methods on quantum computers requires in general exponential resources (i.e.\ the number of gates or the number of iterations of the corresponding quantum circuit grows exponentially in $n$), \cite{wiersema2023optimizing} proposed projecting the Riemannian gradient into smaller dimensional subspaces that scale polynomially in $n$, which in turn yields scalable quantum-computer implementations. 
As already mentioned,  ``dimension reduction'' for efficient quantum-computer implementations comes at the cost of convergence not only to local minima of the cost function but also to
artificial minima induced by projections of the full Riemannian gradient into fixed subspaces.
These artifacts can be avoided by ``random projections'' of the Riemannian gradient \cite{magann2023randomized}.
The random projections can be implemented on quantum computers as depicted in Figs.~\ref{fig:intro1} and~\ref{fig:intro2}. 
More precisely, an efficient quantum-computer implementation of these random 
directions can be achieved through exact \emph{and} approximate unitary 2-designs \cite{dankert2009exact}.

\begin{figure}[htb]
\centering
\includegraphics[width=0.85\textwidth]{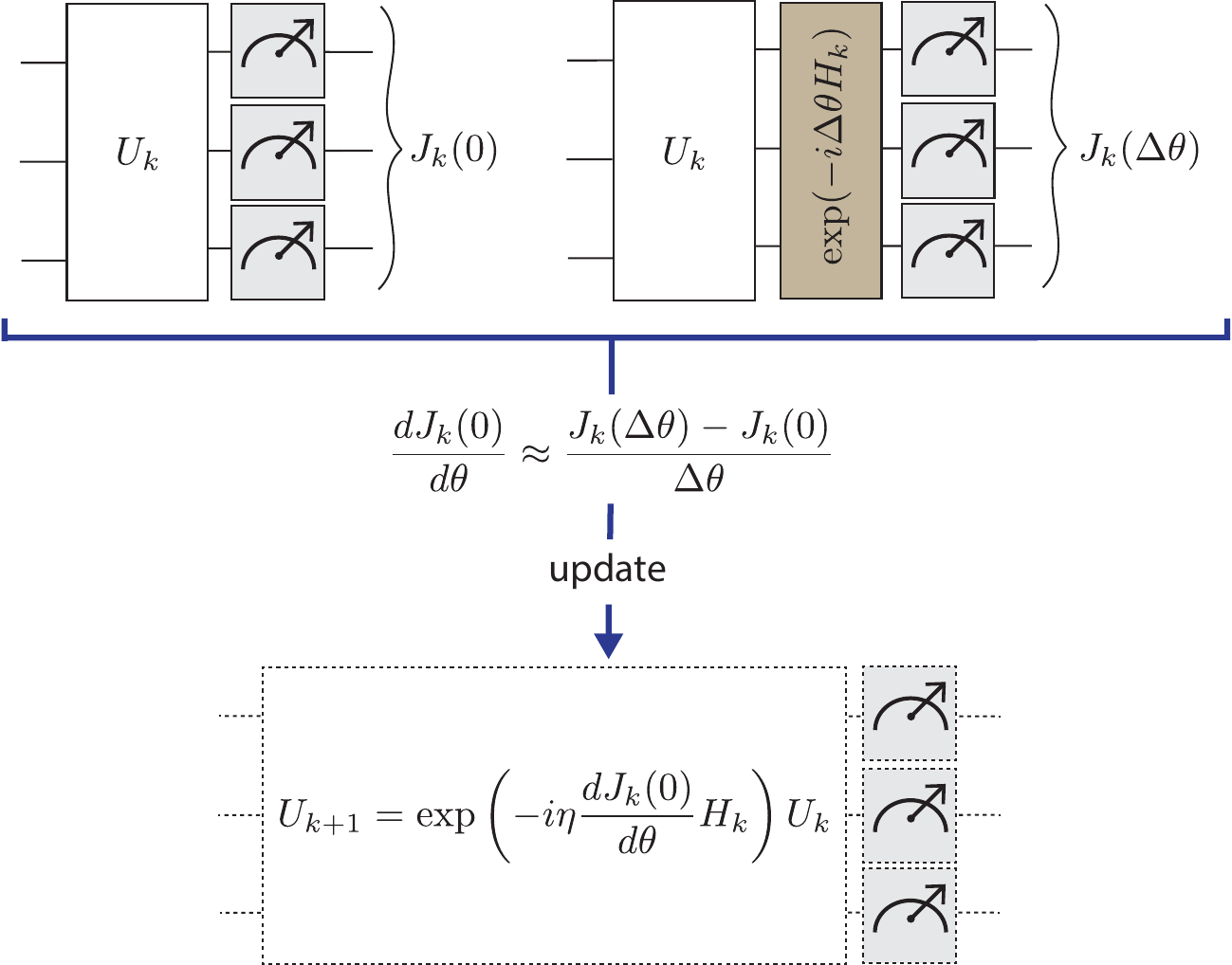}
\caption{Schematic implementation 
of the discretized Riemannian gradient scheme by a quantum algorithm in case the {manifold is $M=\SU(2^n)$.
Thus} the dimension of the problem scales exponentially with the number $n$ of quantum particles {(qubits)}. 
Each quantum circuit {$U_{k} \in \SU(2^n)$} is adaptively turned into $U_{k+1}=\exp\left( \frac{-i\eta\, dJ_{k}(0)}{d\theta}H_{k}\right)U_{k}$ by moving in a random tangent-space direction $iH_{k}$ whose corresponding projection $\frac{dJ_{k}(0)}{d\theta}=\langle \text{grad}J(U_{k}),iH_{k}U_{k}\rangle$ is measured on a quantum device.
Here $\eta$ denotes the step size and {$J_{k}(\theta)$ 
the cost function $\theta \mapsto J_k(\theta) := \mathrm{tr}(A\exp(-i \theta H_{k})U_{k}\rho U_{k}^{*} \exp(i \theta H_{k}))$} of the $k$-th {iteration.}
The {measurement
of $J_k(\theta)$} is done on the quantum computer, and the statistics of repeated measurements
{to obtain the expectation value is done on a classical computer}.  
In this way a randomized gradient flow as in Fig.~\ref{fig:intro1} is implemented.
In this work we show that for sufficiently small $\eta$ such a (discretized) flow converges almost surely to the global optimum.\\\\}
\label{fig:intro2}
\end{figure}

{To continue the mathematical perspective of `ground state' optimization,
i.e. of finding the smallest eigenvalue (`ground state energy') of a Hermitian operator $A$ (Hamiltonian) that describes the energy of the system,} we leave the general notation $f:M\to\mathbb R$ above and specialize to the cost function over the unitary group (or orbit) $J: \U(d)\to\mathbb R$ we aim to minimize. 
For finding the {smallest eigenvalue or, more generally, the} smallest expectation value of $A$ with respect to the initial state $\rho$, the cost function is given by 
\begin{equation}\label{eq:cost-J}
J(U):=\operatorname{tr}(A U\rho U^*)=:\langle A, \Ad_U(\rho)\rangle, 
\end{equation}
where $\rho$ is a density matrix representing the initial state of the system and $U$ is a unitary transformation that describes a quantum circuit (writing $U^*$ for the complex conjugate transpose).

By choosing an appropriate eigenbasis of $A$, we may assume that $A$ is diagonal with eigenvalues in non-increasing order. 
The function we want to optimize is of the form 
of Eqn.~\eqref{eq:cost-J}.
Many properties of this function can be found in~\cite{Duistermaat83} in the more general setting of semisimple Lie groups. We recall the relevant properties here, adapted to our setting:

\medskip
\begin{proposition}
Assume that $\rho$ and $A$ are diagonal\,\footnote{This amounts to a shift in the function {$J$ which results from an appropriate choice of $U := U_0^*V_0$ 
and the identity $\operatorname{tr}(A (U_0^*V_0)\rho (U_0^*V_0)^*) = \operatorname{tr}(U_0 A U_0^* V_0\rho V_0^*)$}.}. 
Then the following hold.
\begin{enumerate}[(i)]
\item \label{it:critical} $U$ is a critical point of $f$ if and only if $[A,\Ad_U(\rho)]=0$. Hence, the critical set of {$J$} is equal to 
$$
\SU(d)_A \,\mc S_d\, \SU(d)_\rho,
$$
where $\mc S_d$ is the set of $d$ by $d$ permutation matrices {and $\SU(d)_A$, $\SU(d)_\rho$ denote the stabilizer subgroups for the action of conjugation}.
In particular, it is a disjoint union of finitely many compact connected submanifolds.
\item \label{it:hessian} The Hessian at a critical point $U$ is determined by
$$
\frac{d^2}{dt^2}\langle A,\Ad_{e^{itH}}(\rho)\rangle
=
\sum_{i>j} 2(a_i-a_j)(\rho_i-\rho_j)|\langle i|\Ad_U(H)|j\rangle|^2,
$$
and in particular {$J$} is Morse--Bott.\footnote{Note that the function {$J$} is never Morse, since the maximal torus in $\SU(d)$ stabilizes $A$ and $\rho$.}
\item \label{it:optima} There is only one local minimal (resp. maximal) value.
\end{enumerate}
\end{proposition}

\begin{proof}
\ref{it:critical}: See Lem.~1.1 and Props.~1.2 and~1.3 of~\cite{Duistermaat83}. 
\ref{it:hessian}: See Prop.~1.4 and Cor.~1.5 of~\cite{Duistermaat83}. 
\ref{it:optima}: See Rmk.~1.6 of~\cite{Duistermaat83}. 
\end{proof}

A key tool for the experimental implementation of the randomized gradient descent algorithm on quantum hardware are unitary $t$-designs, whose purpose is to efficiently implement random unitaries.

\medskip
\begin{definition}[Thm.\ 3 in \cite{GrAuEi07}]
\label{def:t-design}
A unitary representation $\pi:G \to\U(d)$, $g \mapsto \pi(g) =: U_g$ of a finite group $G$ is called a \emph{unitary $t$-design} with $t \in \mathbb{N}_0$ if one of the following equivalent conditions is satisfied:
\begin{enumerate}
\item[(a)]
For all polynomials $p \in \mathbb{C}_{t}[Z,\bar{Z}]$ one has the equality
\begin{equation*}
    \frac{1}{|G|}\sum_{g \in G} p(U_g,\bar{U}_g) =
\int_{\mathrm U(d)} p(U,\bar{U}) dU\,,
\end{equation*}
where $\mathbb{C}_t[Z,\bar{Z}]$ denotes the set of all polynomials in $Z := (z_{ij})_{i,j = 1, \dots d}$ and $\bar{Z} :=(\bar{z}_{ij})_{i,j = 1, \dots d}$, which are $t$ homogeneous in $Z$ and $\bar{Z}$., i.e.%
\footnote{This is equivalent to  $p(\lambda Z, \bar\lambda\bar{Z})=|\lambda|^{2t} p(Z,\bar{Z})$.}
with 
$p(\lambda Z,\mu\bar{Z})=\lambda^t\mu^t p(Z,\bar{Z})$.
\item[(b)]
For all $H \in \mathbb{C}^{d^t \times d^t}$ one has the equality
\begin{equation*}
\frac{1}{|G|}\sum_{g \in G}
\big(U_g \otimes \dots \otimes U_g\big) H
\big(U_g \otimes \dots \otimes U_g\big)^* = \int_{\mathrm U(d)} 
\big(U \otimes \dots \otimes U\big) H
\big(U \otimes \dots \otimes U\big)^* dU\,.
\end{equation*}
\item[(c)]
One has the equality
\begin{equation*}
    \frac{1}{|G|}
    \sum_{g \in G}
    \big(U_g \otimes \dots \otimes U_g\big) \otimes \big(\bar{U}_g \otimes \dots \otimes \bar{U}_g\big) 
    = \int_{\mathrm U(d)} 
    \big(U \otimes \dots \otimes U\big)
    \otimes \big(\bar{U} \otimes \dots \otimes \bar{U}\big) dU\,.
\end{equation*}
\end{enumerate}
For $t = 2$ one has the further equivalence \medskip
\begin{enumerate}
\item[(d)]
The tensor square representation $\pi \otimes \pi: G \to \U(d) \otimes \U(d)$, $g \mapsto U_g \otimes U_g$ acting on
$\mathbb{C}^d \otimes \mathbb{C}^d$ has exactly two irreducible components, namely $\mathrm{Sym}(\mathbb{C}^d \otimes \mathbb{C}^d)$ and
$\mathrm{Alt}(\mathbb{C}^d \otimes \mathbb{C}^d) := \mathbb{C}^d \wedge \mathbb{C}^d$.
\end{enumerate}
\end{definition}

\medskip
\begin{remark}
A  more general notion of $t$-designs  \cite{DankertMsc05, dankert2009exact} (see
also \cite{GrAuEi07})
allows any finite subset of $\U(d)$ which satisfies (a), (b) or, equivalently,(c). However, since almost all known examples are constructed via group representations, we focus here on this restricted approach which is sometimes called group design. Moreover, the reader should note 
the following to facts: (i) One can  assume without loss of generality that $\pi$ is faithful (i.e., one-to-one) because it is
straightforward to show that $[g] \mapsto \pi(g)$, $[g] \in G/\operatorname{ker} \pi$
yields a $t$-design whenever $\pi$ is a $t$-design.
(ii) Every $t$-design is also $t'$-design for 
$t' \leq t$. This follows readily from
condition (b) by choosing $H$ of the form
$H' \otimes I_n \otimes \cdots \otimes I_n$.
\end{remark}
\medskip 

Now we need the following technical lemma:

\medskip
\begin{lemma}
\label{lem:t-design-1}
Let $\pi:G \to \U(d)$ be any unitary representation. Then $\pi$ yields a $2$-design in the sense of Def.~\ref{def:t-design} if and only if
$\pi$ acts irreducibly on $i \mf{su}(n)$
via conjugation, i.e.~the representation 
$\hat{\pi}:G \to U\big(i \mf{su}(d)\big)$,        $\pi(g)\,H := U_g H U_g^*$
is irreducible.
\end{lemma}

\begin{proof}
First, $\hat{\pi}$ obviously ``extends''
to $\tilde{\pi}: G \to U\big(\mathbb{C}^{d\times d}\big)$, $\tilde{\pi}(g)\,A := 
U_g A U_g^{-1}$. Then a straightforward computation shows that the representations $\tilde{\pi}$ and $g \mapsto U_g \otimes \bar U_g$ are equivalent and thus for simplicity we will use the same symbol 
for both representations.
Next, we investigate the commutant of the representations $U_g \otimes \bar U_g$ and $U_g \otimes U_g$ in 
$\mathbb{C}^{d\times d} \otimes \mathbb{C}^{d\times d}\cong \mathbb{C}^{d^2 \times d^2}$,
i.e.~we are interested in the solutions $Z$ of
\begin{equation}
\label{eq:comm-1}
    (U_g \otimes \bar U_g) Z = Z (U_g \otimes \bar U_g) \quad \text{for all} \quad g \in G\,,
\end{equation}
and
\begin{equation}
\label{eq:comm-2}
    (U_g \otimes U_g) Z = Z (U_g \otimes U_g) \quad \text{for all} \quad g \in G\,,
\end{equation}
respectively. Eq.~\eqref{eq:comm-1} and \eqref{eq:comm-2} are equivalent to $\Ad_{U_g \otimes \bar U_g} (Z) = Z$ and $\Ad_{U_g \otimes U_g} (Z) = Z$, respectively. Finally, a tedious but straightforward computation shows that the partial transposed operator $\Phi: \mathbb{C}^{d \times d} \otimes \mathbb{C}^{d \times d} \to \mathbb{C}^{d \times d} \otimes \mathbb{C}^{d \times d}$, $A \otimes B \mapsto A \otimes B^T$ yield an intertwining map for
$\Ad_{U_g \otimes \bar U_g}$ and $\Ad_{U_g \otimes U_g}$, i.e.
$$\Ad_{U_g \otimes \bar U_g} \circ\, \Phi = \Phi \circ  \Ad_{U_g \otimes U_g}\,.$$
This implies $\Phi$ maps the commutator of 
$\tilde{\pi}$ to the commutator of 
$\pi \otimes \pi$ and therefore the dimensions of commutators (and consequently the number of irreducible subspaces) coincide. Thus $\mathrm{Sym}(\mathbb{C}^d \otimes \mathbb{C}^d)$ and
$\mathrm{Alt}(\mathbb{C}^d \otimes \mathbb{C}^d)$ are the only irreducible subspaces of $\pi \otimes \pi$ if and only if $\mathbb{C} \, I_d$ and $\mathfrak{sl}_0(d)$ are the only irreducible subspaces to $\tilde{\pi}$ (or, equivalently, $i\mathfrak{su}(d)$ is the only irreducible subspace of $\hat{\pi}$).
\end{proof}

\begin{remark}
Note that the above result does not mean that the tensor square representation $\pi \otimes \pi$ and $\tilde{\pi}$ are equivalent and, in fact, they are not --- as easy examples demonstrate. 
A similar result for Lie algebra representations was elaborated on in~\cite{ZZ15} as follow-up to~\cite{ZS11} and in particular to~\cite{CZ11}.
\end{remark}

The following is the main result of this section, showing that unitary $2$-designs satisfy Assumption~\ref{assump}. As a consequence, the main convergence result (Thm.~\ref{thm:almost-sure-convergence}) applies to the problem considered in this section.

\medskip
\begin{proposition}
\label{prop:t-desing-2}
Let $\pi: G \to \mathrm U(n)$ be a unitary $t$-design with $t \geq 2$ and let $H_s\in i\mf{su}(n)$ be a traceless Hermitian matrix.
Then any subset of $\{U_g H_s U_g^*\}_{g \in G}$, where one element is removed, still spans $i\mf{su}(n)$.
\end{proposition}

\begin{proof}
By
Lem.~\ref{lem:t-design-1} we know that
the $\hat{\pi}$-invariant subspace spanned by $\{U_g H_s U_g^* : g \in G\}$ has to coincide with $i\mf{su}(d)$ and thus $\{U_g H_s U_g^* : g \in G\}$ contains a basis of $i\mf{su}(d)$. 
Moreover, by the defining property (b) it is straightforward to see that every $1$-design, and therefore also every $2$-design, has to act irreducibly on $\mathbb{C}^d$. 
Thus we conclude
$$
\sum_{g \in G} U_g H_s U_g^* = 0
$$
because the left hand side of the above equation belongs to the commutant of $\pi$ and has trace zero.
Now if there is some $g_0 \in G$ such that $\{U_g H_s U_g^* : g \in G\,,g \neq g_0\}$ does not span $i\mf{su}(d)$, then $U_{g_0} H_s U_{g_0}^*$ does not lie in the span of the others, and hence $\sum_{g \in G} U_g H_s U_g^*\neq0$, which contradicts the above.
\end{proof}

\medskip
\begin{corollary}
The randomized gradient descent algorithm introduced in Sec.~\ref{sec:algorithm-convergence}, applied to the cost function~\eqref{eq:cost-J} on the unitary group, with randomly chosen directions $\{U_g H_s U_g^*\}_{g \in G}$ as in Prop.~\ref{prop:t-desing-2}, converges almost surely to a global minimum.
Moreover, by continuity, this continues to hold if $U_g$ is an approximate $2$-design, as long as the approximation is good enough.
\end{corollary}

\subsubsection*{On the Orbit}

Since the map $J:\mathrm{SU}(d)\to\mathbb R$, which depends on the initial state $\rho$, is invariant under the stabilizer subgroup of $\rho$, denoted $\mathrm{SU}(d)_\rho$, the map can be factored through the corresponding quotient space. 
Equivalently, one may define the map on the unitary orbit of the initial state $\rho$.
We denote this map by $\bar J:\mathrm{SU}(d)\rho\to\mathbb R$.
Working with this function has some theoretical advantages, such as the reduced dimension of the state space. Moreover, in many cases the function is even Morse:

\medskip
\begin{lemma}
If $A$ is non-degenerate, then the function $\bar J$ is Morse.
\end{lemma}

\begin{proof}
See~\cite[Coro.~3.7, Sec.~4]{Duistermaat83}.
\end{proof}

Despite the nice properties of $\bar J$, we do not use this formulation, since it is not clear how to efficiently sample random unitaries yielding directions in the tangent space of the orbit.

\section{Conclusion, Discussion, and Outlook}

We have analyzed the convergence properties of a recently introduced Haar-randomly projected gradient descent algorithm for cost functions taking the form of a smooth Morse--Bott function (with compact sublevel sets) on a Riemannian manifold.
For making it efficient in quantum optimizations, one can approximate the Haar-random projections via unitary 2-designs.
For both scenarios 
we have proven that (i)
the respective algorithm {\em almost surely escapes saddle points} (Lem.~\ref{lemma:saddle-as-gen})
and 
(ii) it {\em almost surely converges to a local minimum} 
(Thm.~\ref{thm:almost-sure-convergence}). 
Moreover, we have studied the time required by the algorithm to pass a saddle point in a simple two-dimensional setting (Sec.~\ref{sec:numerics}).
Note that unlike adiabatic ground-state preparation strategies (that rest on first preparing some desired initial state to find the ground state of some target Hamiltonian), our randomized Riemannian gradient flow algorithm does not require knowledge of the initial state, as the algorithm converges for almost all initial points. 
Furthermore, as in the quantum setting for ground state problems the critical points just comprise saddles and global extrema,
here our result implies almost sure convergence to the {\em global minimum}.  

However, our approach inherits a key problem already arising without projections: the overall speed of convergence to globally optimal solutions may be slow 
\cite{magann2023randomized},
since it depends on the scaling of the magnitude of the (projected) Riemannian gradient. 
For Riemannian optimizations over the unitary group, 
the gradient magnitude  
converges to its expectation value (being zero), 
while the variance is inversely proportional to the dimension $d=2^{n}$ \cite{mcclean2018barren}. 
Thus the probability for a gradient magnitude  larger than the noise level of a quantum computer decays exponentially with the number of qubits $n$.
Well known as {\em barren plateaux problem} in quantum optimization \cite{mcclean2018barren}, such exponentially flat regions 
constitute a main challenge 
to all variational quantum algorithms 
in high dimensions on quantum computers. 

The algorithmic steps analyzed here are entirely modular. In view of future applications, the (approximate) random projections w.r.t.\ the Haar measure may well be replaced by selections from other problem-adapted measures without sacrificing the convergence properties. In particular, one may wish to select the measures such that the projections of the gradient do not subside in numerical noise prematurely and thus circumvent the notorious barren plateaux problem. Moreover, one may consider random projections into subspaces by techniques of compressed sensing and shadow tomography to better approximate the full gradient flow. Likewise,  approximations with tensor-network methods (such as, e.g., \cite{Verstraete16,Verstraete19}) may be envisaged as long as one remains on a Riemannian manifold.
More specifically, see \cite{Barthel23} for the connection of trotterized MERA tensor networks with Riemannian optimization (extended to quasi-Newton) of VQAs in view of hybrid quantum computing. {Another interesting recent application, which may potentially profit from the random methods described here, is
Riemannian quantum circuit optimization based on matrix product operators as suggested by~\cite{Mendl25}.}

Needless to say, the randomized gradient techniques presented here in this paper lend themselves to be taken over to higher-order quasi Newton methods (like L-BFGS) as described in~\cite{SGDH08} also on quantum computers.

Thus we anticipate that the convergence guaranteed for (randomized) Riemannian gradient flows with respect to Morse--Bott type smooth cost functions will encourage wide application in and beyond quantum optimization.

\section*{Acknowledgements}

E.M. and T.S.H. are supported by the 
{\em Munich Center for Quantum
Science and Technology} (MCQST) and the {\em Munich Quantum Valley} (MQV)
with funds from the Agenda Bayern Plus.
C.A. acknowledges support from the National Science Foundation (Grant No. 2231328) and Knowledge Enterprise at Arizona State University. 

\section*{Conflicts of Interest}

All authors declare that they have no conflicts of interest.

\bigskip

\appendix 

\section*{Appendix}

\section{Some Technical Results}

The probability distributions of $u_N$ and $u_N^2$ are important for understanding the behavior of the algorithm. 
Fortunately, they can be described quite easily using the beta distribution. 
Recall that the p.d.f. of the beta distribution with parameters \mbox{$a,b>0$} is given by $f(x) = \mathrm{const}\cdot x^{a-1}(1-x)^{b-1}$.

\medskip
\begin{lemma} \label{lemma:beta-distr}
If $\mbf u\sim\mathcal U(S^{N-1})$ and $u_N$ is the last coordinate, then
$$
u_N\sim\mathcal B_N := 2\,\mathrm{Beta}\left(\frac{N-1}{2},\frac{N-1}{2}\right)-1,
\quad 
u_N^2\sim\mathcal B_N^2 := \mathrm{Beta}\left(\frac12,\frac{N-1}2\right),
$$ 
where $\mathrm{Beta}(a,b)$ denotes the beta distribution, and hence
$$
\mathbb E[u_N^2] = \frac{1}{N}, \quad
\operatorname{Var}(u_N^2) = \frac{2(N-1)}{N^2(N+2)}.
$$
\end{lemma}

\begin{proof}
Let $Z_1,\ldots,Z_N$ be i.i.d.\ standard normals. 
Then, a uniformly random unit vector can be obtained by normalizing the vector $(Z_1,\ldots,Z_N)\in\mathbb R^N$. 
Considering the last coordinate we see that $u_N\overset{d}{=}Z_N/\sqrt{Z_1^2+\ldots+Z_N^2}$. Hence using a well-known relation between the beta and $\chi^2$ (chi-squared) distributions%
\footnote{If $X\sim\chi^2_a$ and $Y\sim\chi^2_b$ are independent, then $\frac{X_{\phantom{j}}}{X^{\phantom{|}}+\,Y}\sim \mathrm{Beta}\left(\frac{a}2,\frac{b}2\right)$, see~\cite[Sec.~20.8 and 24.4]{Balakrishnan03}.} 
we find
$$
u_N^2 
\overset{d}{=}
\frac{Z_N^2}{(Z_1^2+\cdots+Z_{N-1}^2)+Z_N^2}
\sim 
\mathrm{Beta}\left(\frac12,\frac{N-1}2\right).
$$
Since $u_N$ is symmetrically distributed around $0$, it is easy to compute its p.d.f. starting from that of $u_N^2$, and one obtains the form given above.
The expectation and variance of $u_N^2$ follow easily from the formula for the moments of a beta distributed variable $X\sim\mathrm{Beta}(a,b)$:
$$
\mathbb E[X^k] = \prod_{r=0}^{k-1}\frac{a+r}{a+b+r}.
$$
This concludes the proof.
\end{proof}

\medskip
\begin{lemma} \label{lemma:high-dim} 
For $N\geq5$ it holds that
$$
d_K(\sqrt N\mathcal{B}_N,\mathcal N(0,1)) \leq \frac{1}{N}\,,
$$
where $d_K$ denotes the Kolmogorov distance (the supremum distance between the cumulative distribution functions).
This shows that $u_N^2$ is almost distributed according to the distribution $1/N \chi^2_1$, more precisely
$$
d_K(N (\mathcal{B}_N)^2, \chi^2_1) \leq \frac{2}{N}.
$$
If $\Phi$ denotes the cumulative distribution function of a standard normal distribution, then for any $k\geq 0$ it holds that
$$
\Pr\Big(u^2_N \geq \frac{1}{k^2 N}\Big) \geq 2\Big(1-\Phi(\tfrac{1}{k}) - \frac{1}{N}\Big).
$$
\end{lemma}

\begin{proof}
The first inequality follows from~\cite[p.~20]{Pinelis16v5} and~\cite[Thm.~1.2]{Pinelis15}. It is easy to see that taking the square of both distributions at most doubles the Kolmogorov distance, and this yields the second inequality. Finally we compute
$$
\Pr\big(u^2_N \geq \tfrac{1}{k^2 N}\big)
=
2\big(1-\Pr\big(u_N\leq\tfrac{1}{k\sqrt{N}}\big)\big)
\geq 
2\big(1-\Phi\big(\tfrac{1}{k}\big) - \tfrac{1}{N}\big),
$$
proving the third inequality.
\end{proof}


\end{document}